\documentclass[final]{siamltex}   
\usepackage{verbatim}
\usepackage{amsfonts}
\usepackage{amsmath}

\usepackage{multirow}
\usepackage{graphicx}
\usepackage{amssymb}
\usepackage[boxed]{algorithm}
\usepackage{algorithmicx}
\usepackage{algpseudocode}
\usepackage{verbatim}
\usepackage{color}
 \usepackage{epsfig}

\def\nref#1{(\ref{#1})}

\def\intv{$[\xi, \ \eta]$\,\,} 
\def\tla{\tilde \lambda} 
\def\tlu{\tilde u} 
\def\tstar{\hat T}

\newcommand{\qed}{\nobreak \ifvmode \relax \else
      \ifdim\lastskip<1.5em \hskip-\lastskip
      \hskip1.5em plus0em minus0.5em \fi \nobreak
      \vrule height0.75em width0.5em depth0.25em\fi}

\newcommand{\eq}[1]{\begin{equation}\label{#1}}
\newcommand{\en}{\end{equation}}

\graphicspath{{./FIGS/}}

\def\inv{^{-1}}

\def\IR{\mathbb{R}}
\def\IC{\mathbb{C}}

\def\filtlan{\texttt{Filtlan}}

\def\tlu{\tilde u} 
\def\Sup#1{^{({#1})}}

\title{A Thick-Restart Lanczos algorithm with  polynomial filtering
for Hermitian eigenvalue problems}

\author{Ruipeng Li\thanks{Center  for Applied  Scientific Computing,
    Lawrence  Livermore National  Laboratory,  P. O.  Box 808,  L-561,
    Livermore,   CA  94551   {(\tt{li50@llnl.gov})}.  This   work  was
    performed under the  auspices of the U.S. Department  of Energy by
    Lawrence    Livermore   National    Laboratory   under    Contract
    DE-AC52-07NA27344.}
\and Yuanzhe Xi\thanks{Department of Computer Science \& Engineering,
        University of Minnesota, Twin Cities, MN 55455
({\tt \{saad,yxi\}@cs.umn.edu}). This work was supported by the 
  ``High-Performance Computing, and the Scientific Discovery through
  Advanced Computing (SciDAC) program'' funded by U.S. Department of
  Energy, Office of Science, Advanced Scientific Computing Research
  and Basic Energy Sciences DE-SC0008877.}
\and Eugene Vecharynski\thanks{Computational Research Division,
Lawrence Berkeley National Laboratory,
Berkeley, CA 94720 ({\tt \{evecharynski,cyang\}@lbl.gov}). Partial support for this work was provided through Scientific Discovery through Advanced Computing (SciDAC) program funded by U.S. Department of Energy, Office of Science, Advanced Scientific Computing Research.}
\and Chao Yang\footnotemark[3]
\and Yousef Saad\footnotemark[2]
}

\begin{document}

\maketitle

\begin{abstract}
Polynomial filtering can provide a highly effective means of computing
all eigenvalues of a real symmetric (or complex Hermitian) matrix that
are located in a given interval, anywhere in the spectrum.  This paper  describes 
a  technique for tackling this problem by combining  a Thick-Restart  version  of the  Lanczos
algorithm  with  deflation (`locking')  and  a new type of polynomial filters  
obtained from  a least-squares technique.  The resulting algorithm
can be  utilized in a  `spectrum-slicing' approach whereby a 
very large  number of eigenvalues  and associated eigenvectors  of the
matrix are computed by extracting eigenpairs located in different sub-intervals
independently from one another.
\end{abstract}

\begin{keywords}
Lanczos algorithm, polynomial filtering,
Thick-Restart, deflation, spectrum slicing, 
interior eigenvalue problems
\end{keywords}

\section{Introduction\label{sec:intro}}
The problem of computing a very large number of eigenvalues of a large sparse
real symmetric (or complex Hermitian) matrix 
is common to many applications in the physical sciences.
For example, it arises in the Density Functional Theory (DFT) 
based electronic structure calculations for large molecular systems or solids,
where the number of wanted eigenvalues can reach the order of tens of thousands or more.  
While a number of codes were developed in the past for solving large-scale eigenvalue 
problems \cite{Baglama:2003:AIM:838250.838257,Baker:2009:ASN:1527286.1527287,
Bollhofer2007951,Filtlan-paper, citeulike:1474529,
Hernandez:2005:SSF,Stathopoulos10,Yamazaki:2010:APS:1824801.1824805}, these
have not been designed specifically for  handling the situation when the number of targeted 
eigenpairs is extremely large and when the eigenvalues are located well inside the
spectrum. 
It is only in the last few years that this difficult problem has begun to be addressed
by algorithm developers~\cite{Filtlan-paper,Feast,SakSug03, Ve.Ya.Pa:15}. 

Given an $n \times n$ real symmetric  (or complex Hermitian) matrix $A$,
the   problem addressed  in  this  paper  is
to compute \emph{all} of  its eigenvalues that are
located in a  given interval $[\xi,\ \eta]$,
along with their associated eigenvectors. The given interval  should be
a  sub-interval  of  the  interval  $[\lambda_n,\  \lambda_1]$, where
$\lambda_n$ and $ \lambda_1$ 
are the smallest and largest eigenvalues of $A$, respectively.  
In this setting, two types of problems can be distinguished. The most common  situation
treated so far in the literature is the case when
the interval $[\xi,\ \eta]$ is located
at one end of the spectrum, i.e., the case  when either $\xi = \lambda_n$ 
or  $\eta = \lambda_1$. These are often termed extreme eigenvalue problems. 
Computing all eigenvalues in a given interval is typically 
not an issue in this case. Here most methods will work well. 
The second situation, when   $\lambda_n<\xi<\eta<\lambda_1$,  is harder to solve in
general and is often called an  `interior' eigenvalue problem.  

Being able to  efficiently compute the eigenvalues inside a given interval~$[\xi,\ \eta]$ 
constitutes  a critical ingredient in an approach known as  `spectrum slicing' 
for extracting  a large number of eigenpairs.  Spectrum slicing is 
a divide and conquer strategy in which eigenvalues located
in different sub-intervals are computed
 independently from one another. This is discussed in detail 
in the next section.
%
A common approach to 
obtain the part of spectrum in~$[\xi,\ \eta]$ 
is to apply the Lanczos algorithm or subspace iteration to a transformed matrix $B = \rho(A)$, 
where $\rho$ is either a rational function or a polynomial.

For  extreme intervals,  a standard  Chebyshev
acceleration within a subspace iteration  code is exploited  
in \cite{PARSEC-paper} in DFT
self-consistent field (SCF) calculations.  For interior intervals, the
best known  strategy is based on  the shift-and-invert transformation,
where the Lanczos  algorithm or subspace iteration is applied  to $B =
(A-\sigma I)\inv$,  with the shift  $\sigma$ selected to point  to the
eigenvalues in the wanted interval  (e.g., $\sigma$ can be selected as
the middle of the interval).  The shift-and-invert transformation maps
the eigenvalues of $A$ closest to $\sigma$ to the extreme ones of $B$.
This technique may  be effective in some situations but  it requires a
factorization of the matrix $A-\sigma  I$ which can be prohibitively 
expensive for
large  matrices produced from  3D  models.   In  contrast,  polynomial
filtering  essentially replaces  $(A-\sigma I)\inv  $ by  a polynomial
$\rho(A)$ such that all  eigenvalues of $A$ inside $ [  \xi,\ \eta ] $
are  transformed   into  dominant   eigenvalues  of   $\rho(A)$.   Our
experience  in   electronic  structure  calculations   indicates  that
polynomial  filtering  can  perform quite  well.  

The earlier  paper \cite{Filtlan-paper}  described a  filtered Lanczos
approach for  solving the same problem.  Though the basic idea  of the
present paper is  also based on a combination  of polynomial filtering
and  the   Lanczos  process,  the  two   approaches  have  fundamental
differences. First,
the polynomial filters used in \cite{Filtlan-paper} are
different from those of this paper. They are based on a two-stage
approach in which a spline function, called the base filter, is first selected
and then a polynomial is computed to approximate this base filter.
In the present paper, the filter is a simpler least-squares approximation
to the Dirac delta function with various forms of damping. 

The second difference is that the projection method used in  \cite{Filtlan-paper} is the Lanczos
algorithm with partial reorthogonalization \cite{simon1984lanczos} and no restart.
In contrast, the present paper
uses the Lanczos algorithm with a combination of explicit deflation (`locking') 
and implicit restart. A subspace iteration approach is also considered.
In essence, the projection methods used in this paper are geared toward
a limited memory implementation. The choice of the filters 
puts an emphasis on  simplicity as well as improved 
robustness relative to   \cite{Filtlan-paper}.

The paper is organized as follows. 
Section~\ref{sec:Slicing} provides a high-level description of the spectrum slicing strategy
for computing a large subset of eigenvalues,
which is the main motivation for this work. Section \ref{sec:filters} introduces a least-squares viewpoint for deriving polynomial filters used in the eigenvalue computation. Section \ref{sec:filtlan} discusses how to efficiently combine the restarted Lanczos algorithm with polynomial filtering and deflation. Numerical examples are provided in Section \ref{sec:num} and the paper ends with concluding remarks in Section \ref{sec:con}.

\section{Motivation: Spectrum slicing}\label{sec:Slicing}
The algorithms studied in this paper are part of a bigger project
to develop a parallel package named EVSL (Eigen-Value Slicing Library) for
 extracting very large numbers of eigenvalues and their associated
eigenvectors of a matrix. The basic methodology adopted in EVSL is 
a divide-and-conquer approach known as \emph{spectrum slicing.}
\subsection{Background}
The principle of spectrum slicing is 
conceptually  simple. It consists of dividing the overall interval containing the spectrum into small sub-intervals and then computing eigenpairs 
in each sub-interval  independently. 

For this to work, it is necessary to 
develop a procedure that is 
able to  extract all eigenvalues in a given arbitrary small interval.
Such a  procedure must satisfy two important requirements.
The first is that the eigenpairs
in each sub-interval under consideration 
are to be computed independently from any of the other sub-intervals. 
The procedure should be as oblivious as possible to any other calculations. The only 
possible exception is that we may have  instances where
checking for orthogonality between nearby  pairs will be warranted. 
The other requirement is that the procedure under consideration should not miss any eigenvalue.

The idea of spectrum slicing by polynomial filtering is illustrated
in Figure~\ref{fig:filters}.
In this approach, the spectrum is first linearly
transformed into the interval [-1, 1]. This transformation is necessary
because the polynomials are often expressed in Chebyshev bases.
The interval of interest is then split into $p$ sub-intervals ($p=3$ in the illustration). 
In each of the sub-intervals we select a filter polynomial of a certain
degree so that eigenvalues within the sub-interval are amplified. 
In the illustration shown in Figure~\ref{fig:filters}, the 
polynomials are of degree 20 (left), 30 (middle), and 32 (right).
 High intersection points of the curves delineate the final sub-intervals
used. 

\begin{figure}[tbh] 
\centering
\centerline{\includegraphics[width=0.66\textwidth]{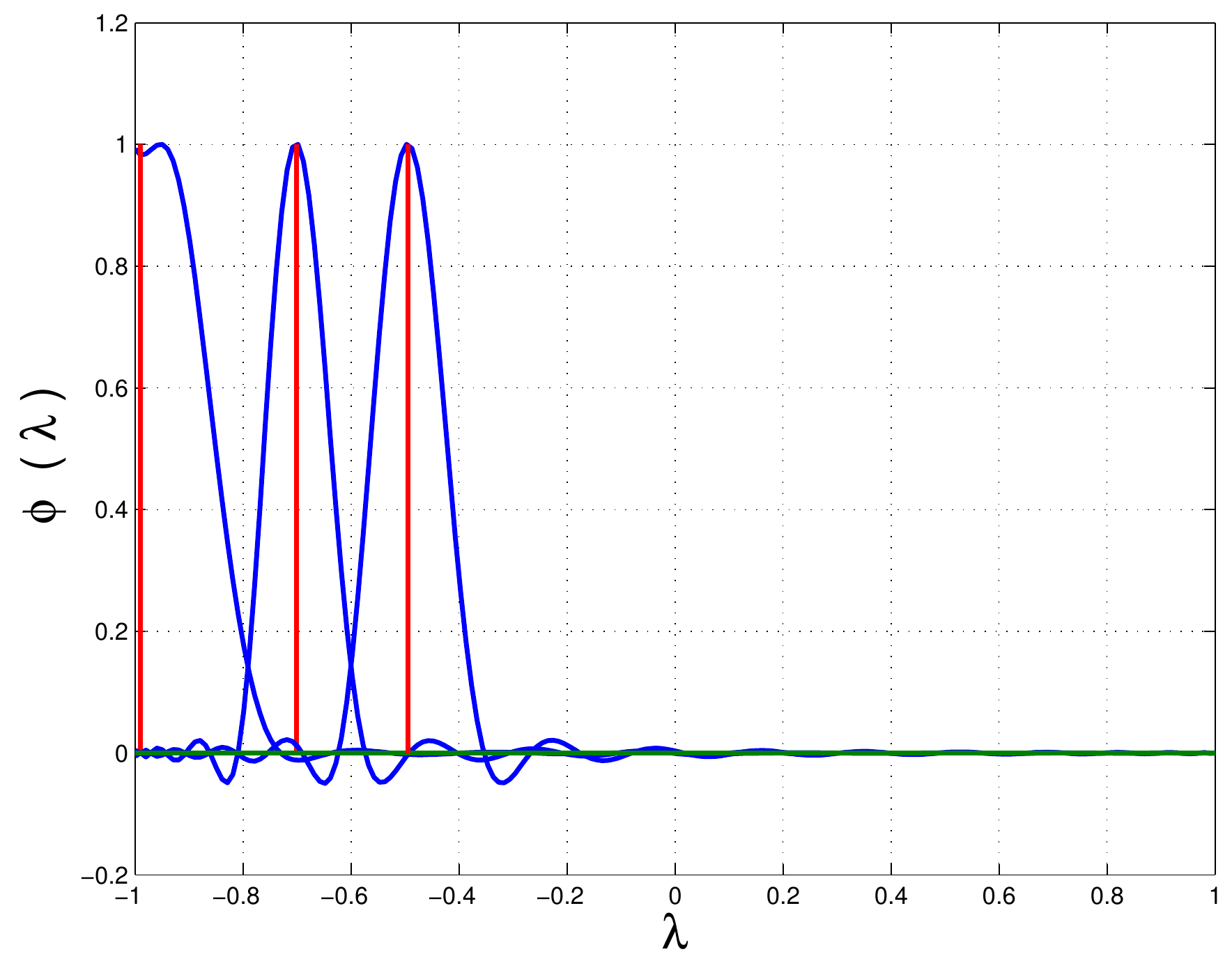}} 
\caption{Polynomial filters for 3 different slices.}\label{fig:filters}
\end{figure}


\subsection{Slicing strategies}
The simplest way to slice a target interval is to divide it uniformly
into several sub-intervals of equal size.  
However, when the distribution of eigenvalues is nonuniform, 
some intervals may contain many more eigenvalues than others. 
An alternative is required  
so that the cost of computing 
eigenvalues within each sub-interval is not excessive.

A  better slicing  strategy is  to divide  the interval  based on  the
distribution of  eigenvalues. This can  be done by exploiting algorithms 
for computing density of states (DOS)
\cite{LinYangSaad2013-TR}. 
We can use Newton's method or a form of
bisection to divide the interval in such a way that each
sub-interval  contains   roughly  the  same  number   of  eigenvalues.
However, when the eigenvalues are  not uniformly distributed, the size
of each partitioned sub-interval can  be quite different. As a result,
we  may  have  to  use  filter polynomials  with  varying  degrees  on
different slices.  One advantage of  this slicing strategy is  that we
can  ensure that  roughly the  same amount  of memory  is required  to
compute  eigenpairs  within  each  sub-interval.  This  is  especially
important for a parallel computational environment.

The optimal number of slices in which to partition the spectrum 
depends  on a number of factors such the 
efficiency  of matrix-vector  multiplications, the 
total number of eigenpairs to be computed, etc. In a parallel 
computing environment, it also depends on the number of processors available.

\subsection{Parallel strategies}
If a given interval contains many eigenpairs, we 
divide it into a number of sub-intervals and map
each sub-interval to a group of processors, so that eigenvalues
contained in different sub-intervals can be computed in parallel.
This strategy forms the basic configuration of parallel
 spectrum slicing and is illustrated in Figure \ref{fig:evsl}.
A division of the initial interval into sub-intervals containing 
 roughly the same number of eigenvalues will
prevent the orthogonalization and Rayleigh-Ritz procedure from becoming 
a bottleneck. It will also 
limit the amount of work (hence the granularity) associated with
each concurrent subtask.  When the number of sub-intervals is much larger
than the number of processor groups, dynamic scheduling should be used to 
map each sub-interval to a processor group and launch the computation 
for several sub-intervals in parallel. 
Parallelism across slices
constitutes only one of the levels of parallelism. Another level 
corresponds to the matrix-vector operations which can be treated with graph
partitioning.

\begin{figure}[tbh] 
\centering
\centerline{\includegraphics[width=0.66\textwidth]{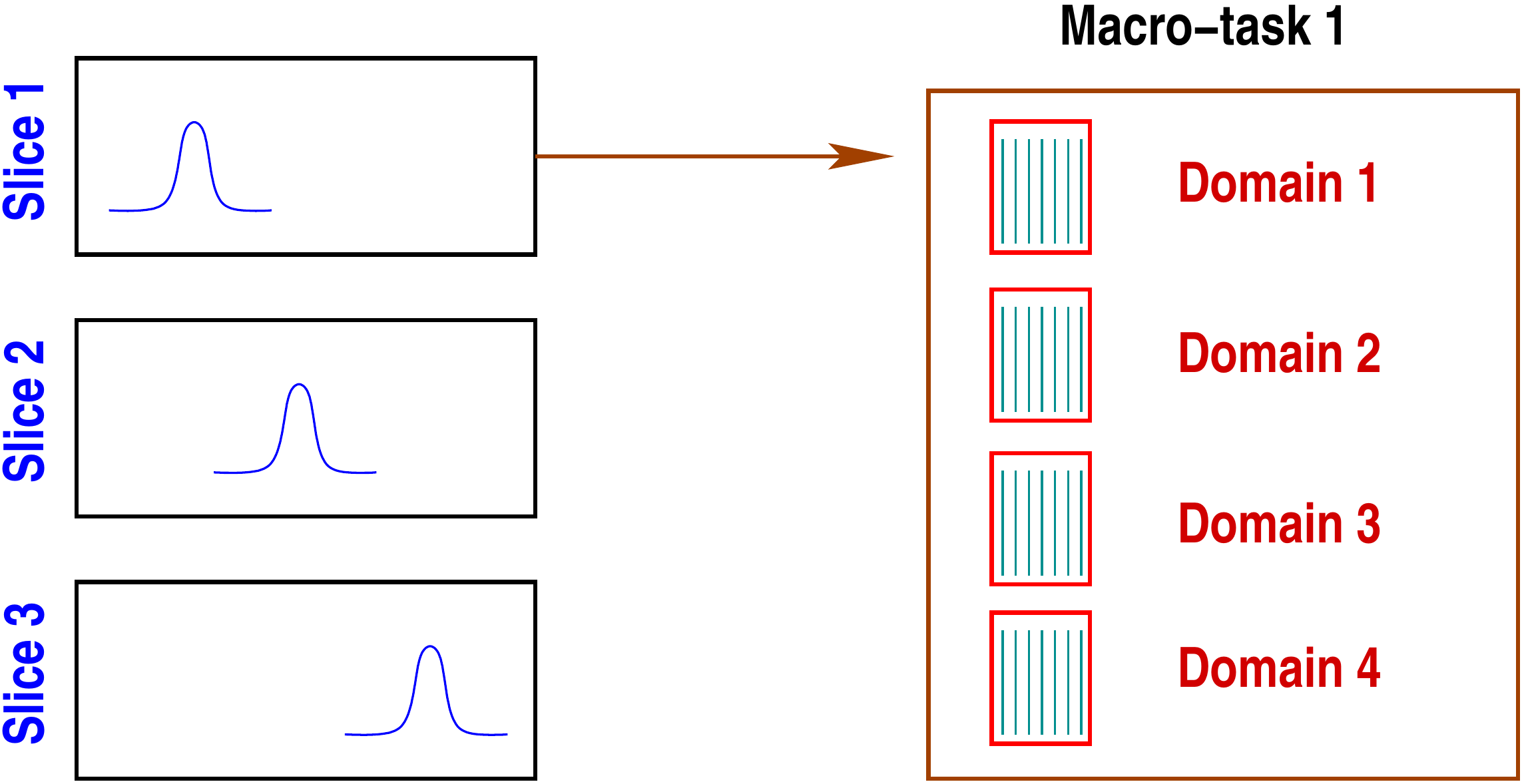}}
\caption{The two main levels of parallelism in \emph{EVSL}.}
\label{fig:evsl}
\end{figure}



\section{Least-squares polynomial filters}\label{sec:filters} 
The article  \cite{Filtlan-paper} relied on polynomial filters developed in
\cite{saad-FILT}. These polynomials are computed as least-squares approximations
to a base filter, typically a spline function. The base filter transforms 
eigenvalues in the desired interval to  values that are close to one and 
those outside the interval to values that are close to zero. The advantage of this
procedure is that it is very flexible since the base filter can be selected in
many different ways to satisfy any desired requirement. On the other hand,
the procedure is somewhat complex.
In EVSL we will only use one type of filter namely, Chebyshev polynomial approximation to the 
Dirac delta function. 

\subsection{Approximating the Dirac delta function}\label{sec:deltFun} 
Since Chebyshev  polynomials are  defined over the  reference interval
$[-1, \ 1]$, a linear transformation  is needed to map the eigenvalues
of a general Hermitian matrix $A$  to this reference interval. This is
achieved by the  following transformation: \eq{eq:mapping} \widehat{A}
=  \frac{A -  c I}{d}  \qquad \mbox{with}\quad  c =  \frac{\lambda_1 +
  \lambda_n}{2}, \quad  d =  \frac{\lambda_1 - \lambda_n}{2}.   \en In
other  words, all  work will  be performed  on the  transformed matrix
$\widehat{A}$ which now  has its eigenvalues in  the interval $[-1, \ 1]$.
In   practice,   the   maximum   ($\lambda_1$)   and   the
minimum~($\lambda_n$) eigenvalues of  $A$ can be replaced  by an upper
bound  $\tilde{\lambda}_1$  and   a  lower  bound  $\tilde{\lambda}_n$
obtained by adequate perturbations of  the largest and smallest eigenvalues 
obtained from a few steps  of  the  standard  Lanczos
algorithm~\cite{zstc:chebyshev06}.

Assume that  the matrix $A$ is  thus linearly transformed so  that its
eigenvalues are in $[-1, \ 1]$ and  let $\gamma$ be the value on which
the Dirac delta function is to be centered. Then, apart from a
scaling by a constant,  a formal expansion of
this  Dirac delta  function 
  \eq{eq:Cheb0} \rho_k(t)  = \sum_{j=0}^k
\mu_j T_j(t), \en with
\begin{eqnarray}
\label{eq:mu}
\mu_j &=& \left\{
\begin{array}{ll}
\frac{1}{2} & \mbox{if} \  j = 0\cr 
\cos(j \cos\inv (\gamma) ) & \mbox{otherwise}  
\end{array} \right.
\end{eqnarray}
where $T_j(t)$ is the Chebyshev polynomial of the first kind of degree $j$.
Consider the normalized sequence of Chebyshev polynomials:
\begin{eqnarray}
\label{eq:ncheby}
\tstar_j &=& \left\{
\begin{array}{ll}
T_j/\sqrt{\pi} & \mbox{if} \  j = 0\cr 
T_j/\sqrt{\pi/2} & \mbox{otherwise}  
\end{array} \right.
\end{eqnarray}
which are orthonormal in that 
$\langle \tstar_i \,, \tstar_j \rangle_w = \delta_{ij}$, 
where $\langle\cdot \,,\cdot\rangle_w$ represents the Chebyshev $L^2$ inner product, and $\delta_{ij}$ is the Kronecker delta function.
Then, the formal expansion of $\delta_\gamma$ is
$\delta_\gamma \approx \sum_{j=0}^k \hat \mu_j \hat T_j $, where
$\hat \mu_j = \langle \tstar_j \,, \delta_\gamma \rangle_w$.
Making the change of variables $t = \cos \theta $ and setting
$\theta_\gamma = \cos\inv (\gamma)$, we get
\[ 
\hat \mu_j = \sqrt{\frac{2 - \delta_{j0}}{\pi}} \
 \int_0^\pi \cos (j \theta) \delta_{\theta_\gamma} d \theta 
= \sqrt{\frac{2 - \delta_{j0}}{\pi}} \
 \cos ( j \theta_\gamma ) . \]
Thus, it can be seen that 
 $\sum \hat \mu_j \hat T_j(t)  = \frac{2}{\pi} \sum \mu_j T_j(t) $ 
with $\mu_j$ defined by \nref{eq:mu}.

The biggest attraction of the above expansion relative to the one used
in \cite{Filtlan-paper} is its simplicity.  It has no other parameters
than the degree  $k$ and the point $\gamma$ on  which the Dirac delta
function is centered.  It may be  argued that it is not mathematically
rigorous or permissible to expand the Dirac delta function, which is a
distribution, into orthogonal  polynomials. Fortunately, the resulting
polynomials obey an alternative  criterion. 

\begin{proposition} 
Let $\rho_k(t) $ be the Chebyshev expansion
defined by (\ref{eq:Cheb0})--(\ref{eq:mu})
and let  $\hat \rho_k(t)$ be the polynomial that minimizes
\eq{eq:ObjAlt}
\| r(t) \|_w
\en 
over all polynomials $r$ of degree $\le k$, such that $r(\gamma) = 1$,
where $\| . \|_w$ represents the Chebyshev $L^2$-norm.  Then 
$\hat \rho_k (t) = \rho_k(t) / \rho_k(\gamma)$.
\end{proposition} 
\begin{proof}
It is known  \cite{poly} that 
the unique minimizer of (\ref{eq:ObjAlt}) can be expressed via the 
Kernel polynomial formula: 
\eq{eq:KernP}
\hat{\rho}_k(t)\equiv \frac{\sum_{j=0}^{k}\tstar_j(\gamma)\tstar_j(t)}{\sum_{j=0}^{k}\tstar_j^2(\gamma)} .
\en
From the expression of   $\tstar_j$ in terms of $T_j$ in 
 \nref{eq:ncheby}, the above equation yields
\[
\hat{\rho}_k(t) = \frac{T_0(\gamma)T_0(t)+\sum_{j=1}^{k}2T_j(\gamma){T}_j(t)}{T^2_0(\gamma)+\sum_{j=1}^{k}2{T}_j^2(\gamma)} = \frac{T_0(t)/2+\sum_{j=1}^{k}T_j(\gamma){T}_j(t)}{T^2_0(\gamma)/2+\sum_{j=1}^{k}{T}_j^2(\gamma)}. 
\]
The numerator is equal to $\rho_k$ defined by (\ref{eq:Cheb0})--(\ref{eq:mu})  
and so the polynomials $\rho_k$  and $\hat \rho_k$ are  multiples of 
one another.
\end{proof}

Least-squares polynomials of this type have been advocated  and studied in the
context of polynomial preconditioning, where $\gamma = 0$;
see \cite[\S 12.3.3]{saad:iter03} and \cite{Saad-pol}.
In this case the polynomial $\hat{\rho}_k(t)$ is 
the `residual polynomial' used for  preconditioning.
A number of results have been established for the special case
$\gamma = 0$ in \cite{Saad-pol}. Here we wish to consider
additional results for the general situation where $\gamma \ne 0$. 


The starting point is the alternative expression \nref{eq:KernP} of the
normalized filter polynomial $\hat \rho_k(t)$. Next, we state  a corollary of 
the well-known Christoffel-Darboux formula on kernel polynomials shown
in \cite[Corollary 10.1.7]{Davis-book}. When written for
Chebyshev polynomials \nref{eq:ncheby} of the first kind, this corollary 
gives the following result.
\begin{lemma} \label{lem:lemCB}
Let $\hat T_j$, $j=0,1,\ldots$ be the orthonormal Chebyshev polynomials of
the first kind defined in~(\ref{eq:ncheby}). Then
\eq{eq:lemCB}
\sum_{j=0}^m  [\hat T_j (t)]^2  = 
\frac{1}{2} \left[
\hat T_{m+1}' (t) \hat T_m (t) - \hat T_{m}' (t) \hat T_{m+1} (t)
\right] . 
\en
\end{lemma} 
This will allow us to analyze the integral of
$\hat \rho_k^2$ with respect to the
Chebyshev weight. 
\begin{theorem} \label{th:NormCB} Assuming $k\ge 1$, the following equalities
hold: 
\begin{eqnarray}
\int_{-1}^1 \frac{[\hat \rho_k(s)]^2}{\sqrt{1-s^2}} ds &=& 
\frac{1}{\sum_{j=0}^k  [\hat T_j (\gamma)]^2} \label{eq:thNorm} \\
&=& 
\frac{2 \pi}{(2k+1)}  \times  
\frac{1}{1 + \frac{\sin (2k+1) \theta_\gamma }{ (2 k+1) \sin  \theta_\gamma }} 
\ ,
\label{eq:thNorm1}
\end{eqnarray}
where  $\theta_\gamma = \cos^{-1} \gamma$.
\end{theorem} 
\begin{proof}
Let $D_k \equiv \sum_{j=0}^k  [\hat T_j (\gamma)]^2 $.
Since the sequence of polynomials $\hat T_j$ is orthonormal, \nref{eq:KernP}
implies that 
\[\| \hat \rho_k \|_w^2 = \left[\frac{1}{D_k}\right]^2 
\sum_{j=0}^k [\hat T_j (\gamma)^2 ] = \frac{1}{D_k} . 
\]
This shows \nref{eq:thNorm}. 
We now invoke  Lemma~\ref{lem:lemCB} to evaluate $D_k$. 
Recall that 
$T_j'(t) =  j \sin(j \theta)/\sin(\theta)$ where $\cos (\theta) = t$.
Then we obtain,  for any $t = \cos \theta $: 
\begin{eqnarray*}
\sum_{j=0}^k [\hat T_j (t)]^2 & = & 
\frac{1}{2} 
\left[
\hat T_{k+1}' (t) \hat T_k (t) 
- \hat T_{k}' (t) \hat T_{k+1} (t)
\right] \\
& = & \frac{1}{\pi} \frac{ (k+1) \sin((k+1)\theta) \cos (k \theta) - 
k\sin(k \theta) \cos( (k+1) \theta) }
{\sin \theta} \\
& = & \frac{k}{\pi} +
 \frac{1}{\pi} 
 \frac{ \sin((k+1)\theta) \cos(k \theta)}{\sin \theta} \\
& = & \frac{k}{\pi} +
 \frac{1}{\pi} 
 \frac{ \sin((k+1)\theta + k \theta) +  \sin((k+1)\theta - k \theta) }
{2 \sin \theta} \\
& = & \frac{2k+1}{2 \pi} +
 \frac{1}{\pi} 
 \frac{ \sin((2 k+1)\theta )}{2 \sin \theta} . 
\end{eqnarray*}
This leads to \nref{eq:thNorm1}, by setting 
 $\theta_\gamma = \cos^{-1} \gamma$, 
 and factoring  the term $2\pi / (2 k +1)$.
\end{proof}

The term $ \sin ((2k+1)\theta) / ((2k+1)\sin \theta)$ 
in the denominator of \nref{eq:thNorm1} 
is related to the Dirichlet Kernel \cite{Davis-book}. It is 
a Chebyshev polynomial  of the second kind of degree $2k$ normalized
so that its maximum value is equal to one. Recall that for these 
polynomials, which are often denoted by $U_m$,
we have $U_m(1) = m+1$ and $U_m(-1) = (-1)^m (m+1)$
and these are the points of largest magnitude of $U_m$. 
Hence $ \sin ((2k+1)\theta) / ((2k+1)\sin \theta) = U_{2k} (t)/(2k+1) \le 1$,
with the maximum value of  one reached.
Just as  important is the behavior of the minimum value which seems 
to  follow a pattern 
 similar to that of other Gibbs oscillation phenomena observed.
Specifically, the minimum value seems to converge 
to the value -0.217... as the degree increases. 
Three plots are shown in Figure~\ref{fig:modDirc} for illustration. 
\begin{figure}[tbh] 
\centering
\centerline{\includegraphics[width=0.56\textwidth]{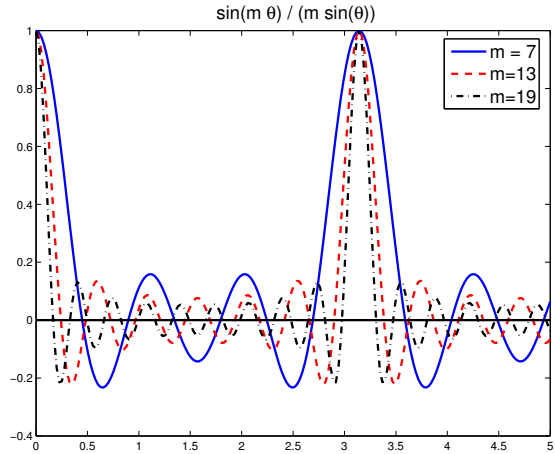}} 
\caption{The function $\sin m\theta/ (m \sin \theta) $ for
$m = 7, 13, 19$}\label{fig:modDirc}
\end{figure}

The above analysis shows that  the integral in \nref{eq:thNorm} decreases 
roughly like $2\pi/(2k+1)$. It can help understand how  the function is 
expected to converge.  
For example, in Section~\ref{sec:degSelect} we will 
exploit the fact that at the (fixed) 
boundaries $\xi, \ \eta$ of the interval the
polynomial value is bound to decrease as the degree increases,  otherwise the 
integral would remain higher than a certain value which would lead to 
a contradiction. 

\subsection{Oscillations and damping}
As is well-known,  expansions of discontinuous functions lead to 
oscillations near the discontinuities  known as \emph{Gibbs oscillations}. 
To alleviate   this behavior it is customary to add damping multipliers 
so that \nref{eq:Cheb0} is actually replaced by
\eq{eq:polPJ} 
\rho_{k} ( t) = \sum_{j=0}^k g_j^{k} \mu_j T_j(t) .  
\en 
Thus, the original expansion coefficients
$\mu_j$ in the expansion \nref{eq:Cheb0}
 are multiplied by smoothing factors $g_j^k$.  These
tend to be quite small  for the  larger values of $j$ that correspond to
 the highly oscillatory terms in the  expansion.
Jackson smoothing, see, e.g., ~\cite{rivlin:approx10,Jay-al},
is the best known  approach. The corresponding  coefficients $g_j^k$ are given 
by the formula
\[
g_j^k=\frac{\sin((j+1)\alpha_k)}{(k+2)\sin(\alpha_k)}+\left(1-\frac{j+1}{k+2}
\right)\cos(j\alpha_k),
\]
where $\alpha_k=\frac{\pi}{k+2}$. More details on this expression
can be seen in \cite{Jay-al}. Not as well known is another
form of smoothing proposed by Lanczos~\cite[Chap. 4]{Lanczos-book} and
referred to as $\sigma$-smoothing. It 
uses the following  simpler damping coefficients instead of $g_j^k$, called 
 $\sigma$ factors by the author:
\[
\sigma_0^k =1; \quad
\sigma_j^k= \frac{ \sin ( j \theta_k ) }{j \theta_k},
j=1,\ldots,k;  \quad \mbox{with} \quad   
\theta_k =\frac{\pi}{k+1} . 
\] 
Figure~\ref{fig:threeFilt} shows an illustration of three filters of degree 20,
one of which is without damping and the others using the Jackson damping and 
the Lanczos $\sigma$-damping, respectively. Scaling is used so that all $3$ polynomials take the same value $1$ at $\gamma$.

\begin{figure}[tbh]  
\centering
\centerline{\includegraphics[width=0.56\textwidth]{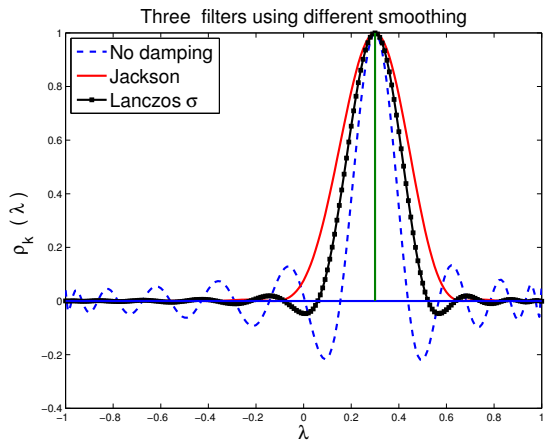}}
\caption{Polynomial filters of degree 20
 using 3 different damping techniques.}\label{fig:threeFilt}
\end{figure}

\subsection{Choosing the degree of the filter}\label{sec:degSelect}
A good procedure based on polynomial filtering  should begin by selecting the
polynomial and this cannot be left to the user. It must be done  automatically, 
requiring only  the sub-interval $[\xi,\ \eta]$ as input, 
and optionally the type of damping to be used. The procedure we currently use
starts with a low degree polynomial (e.g., $k=3$) and then  increases $k$ until
the values of $\rho_k(\xi)$ and $\rho_k(\eta)$ both fall below a certain threshold $\phi$
(e.g., $\phi=0.6$ for mid-interval filters, $\phi=0.3$ for end intervals). For example,
in Figure~\ref{fig:findDeg}, on the left side we would get a degree of
$k=15$ when $\phi=0.3$ and on the right side $k=20$ when $\phi=0.6$.
Once the degree has been selected, a post-processing is carried out to try
to get a `balanced polynomial', i.e., one  whose values at $\xi$ and $\eta$ are the same.
This is discussed next.

\begin{figure}[tbh] 
\centering
\includegraphics[width=0.46\textwidth]{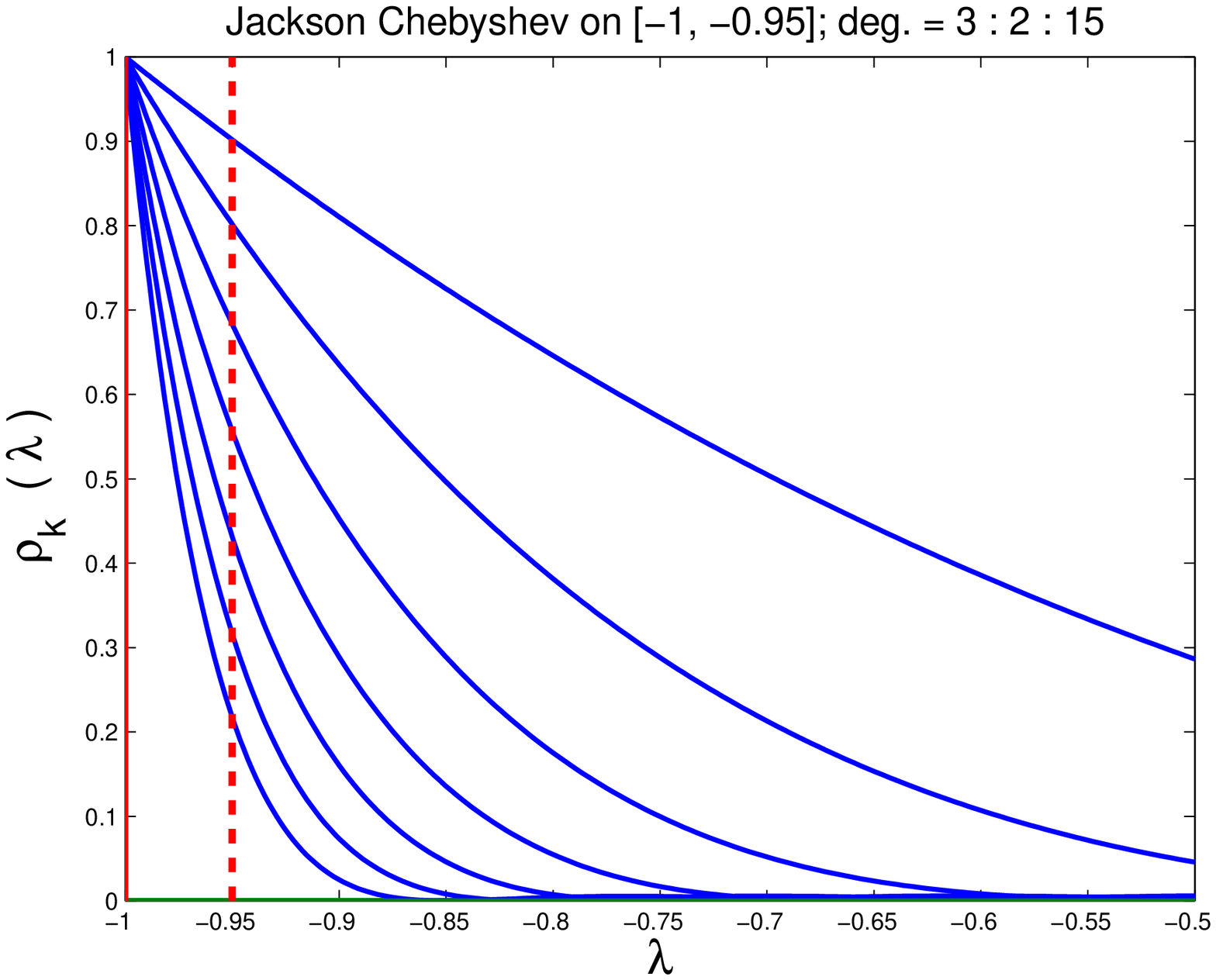}
\includegraphics[width=0.46\textwidth]{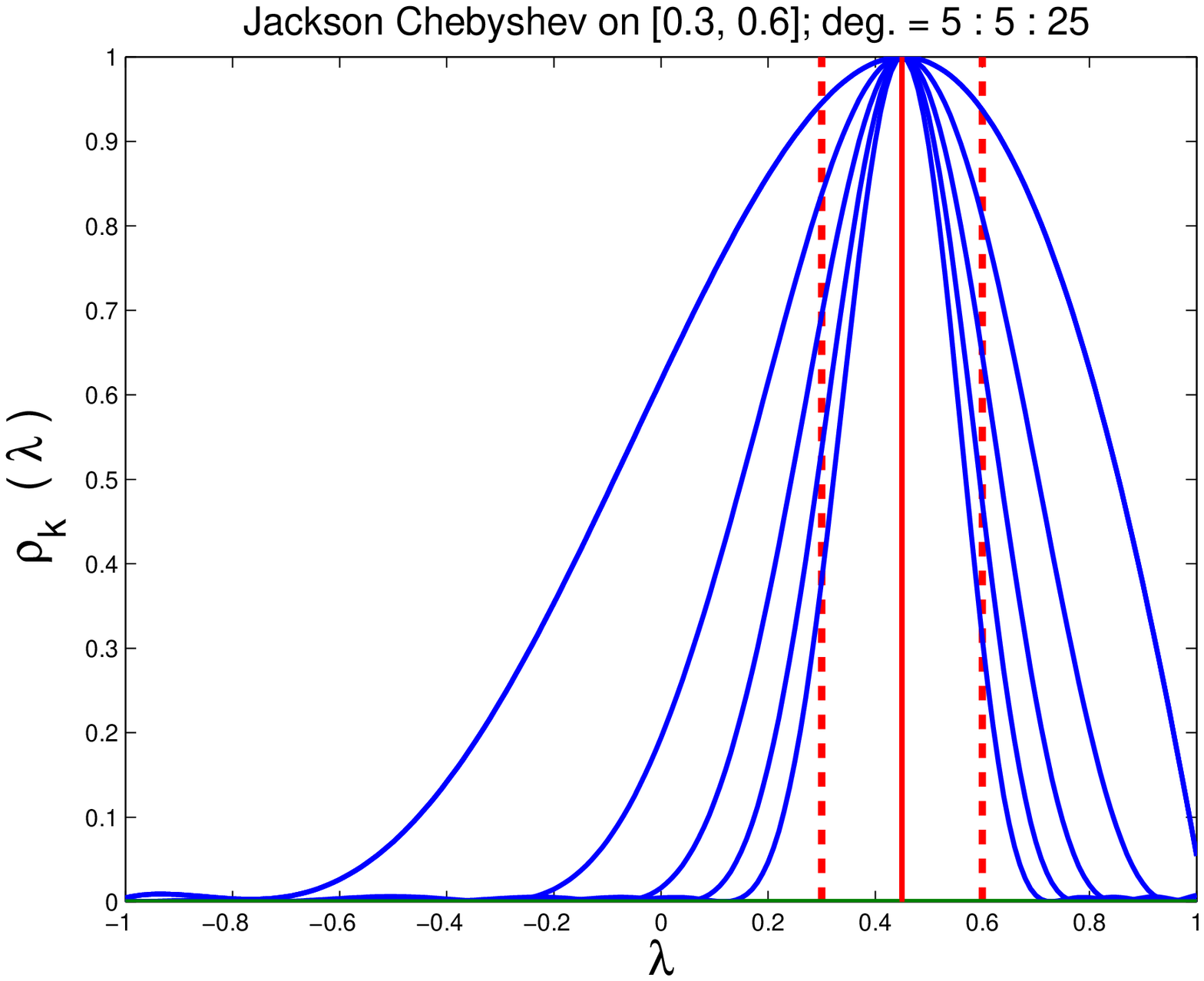}
\caption{Jackson-Chebyshev polynomial filters for a
left-end interval (left) and a middle interval (right).}
\label{fig:findDeg}
\end{figure} 


\subsection{Balancing the filter} 

It is preferable  to have a filter $\rho_k$ whose  values at $\xi$ and
$\eta$ are the same to facilitate the selection of desired eigenvalues
during  the Lanczos  procedure.  For this, a process is
required in which the  center  $\gamma$  of  the Delta  function  is
moved.  This is not a costly process but   its  effects  are
important. For example, it now becomes easy to determine if a computed
eigenvalue  $\theta=\rho_k(\lambda)$  corresponds   to  an  eigenvalue
$\lambda$ that is inside or outside the interval $[\xi, \eta]$.
If  $\phi \equiv \rho_k{(\xi)} = \rho_k{(\eta})$ then: 
\[ \lambda \ \in \ [\xi, \eta] \quad \mbox{iff}\quad \theta\equiv \rho_k(\lambda) \ge \phi . 
\]
Thus, to  find all eigenvalues $  \lambda \ \in  \ [\xi, \ \eta]  $ it
suffices  to find  all eigenvalues  of $\rho_k(\widehat{A})$  that are
greater than or equal to $\phi$.  In the actual algorithm, this
serves as a preselection tool only.  All eigenvalues $\theta_j$'s that
are  above  the  threshold  $\phi$ are  preselected.  Then  the
corresponding  eigenvectors  $\tilde   u_j$'s are computed along with the  
Rayleigh quotients  $\tilde u_j^H  A \tilde  u_j$, which will be  
ignored  if  they do  not  belong  to  $[\xi, \eta]$.    
Additional  details   will   be   given  in the full algorithm described 
in  Section \ref{sec:all}.

\begin{figure}[htb]  
\centering
\includegraphics[width=0.46\textwidth]{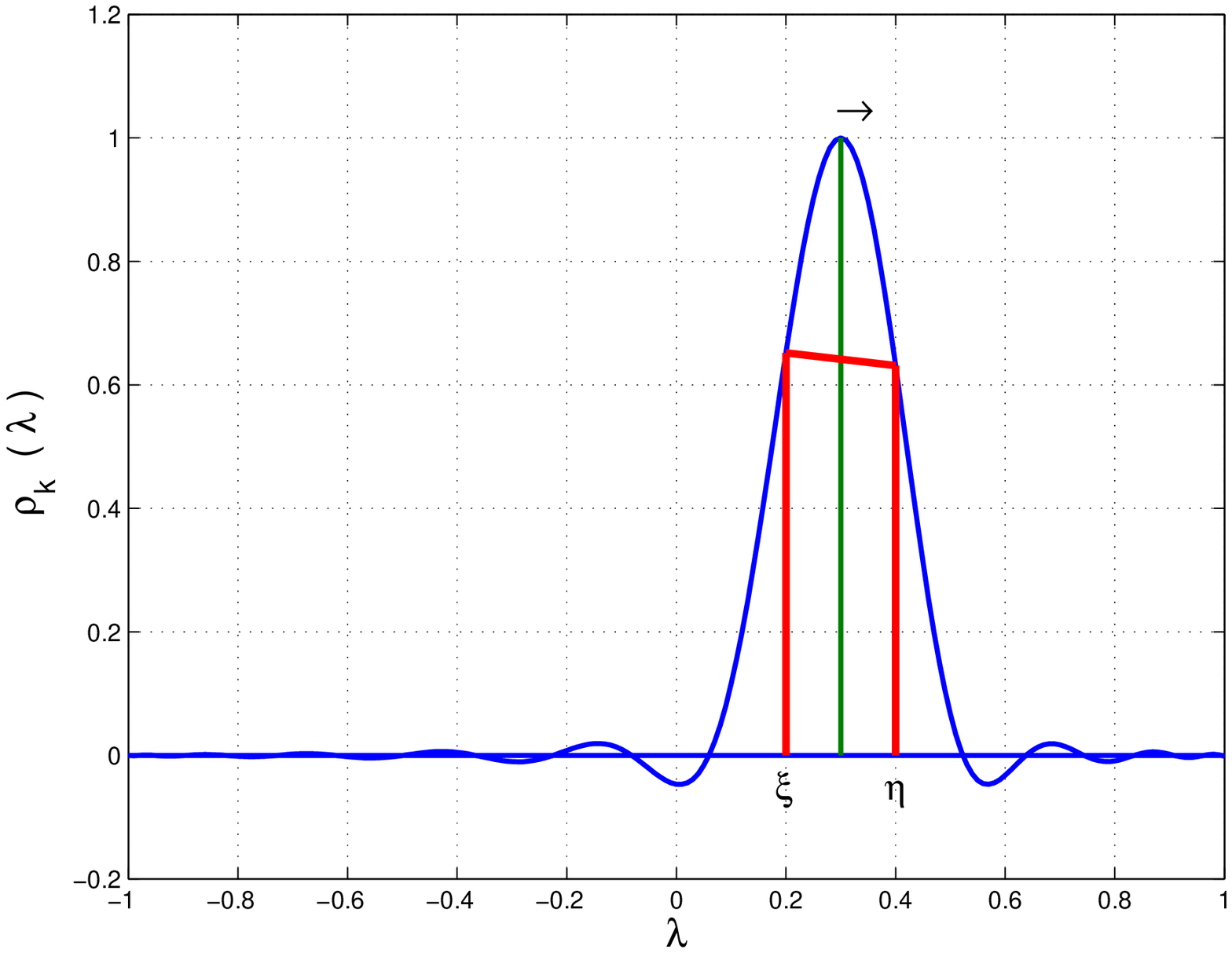}
\includegraphics[width=0.46\textwidth]{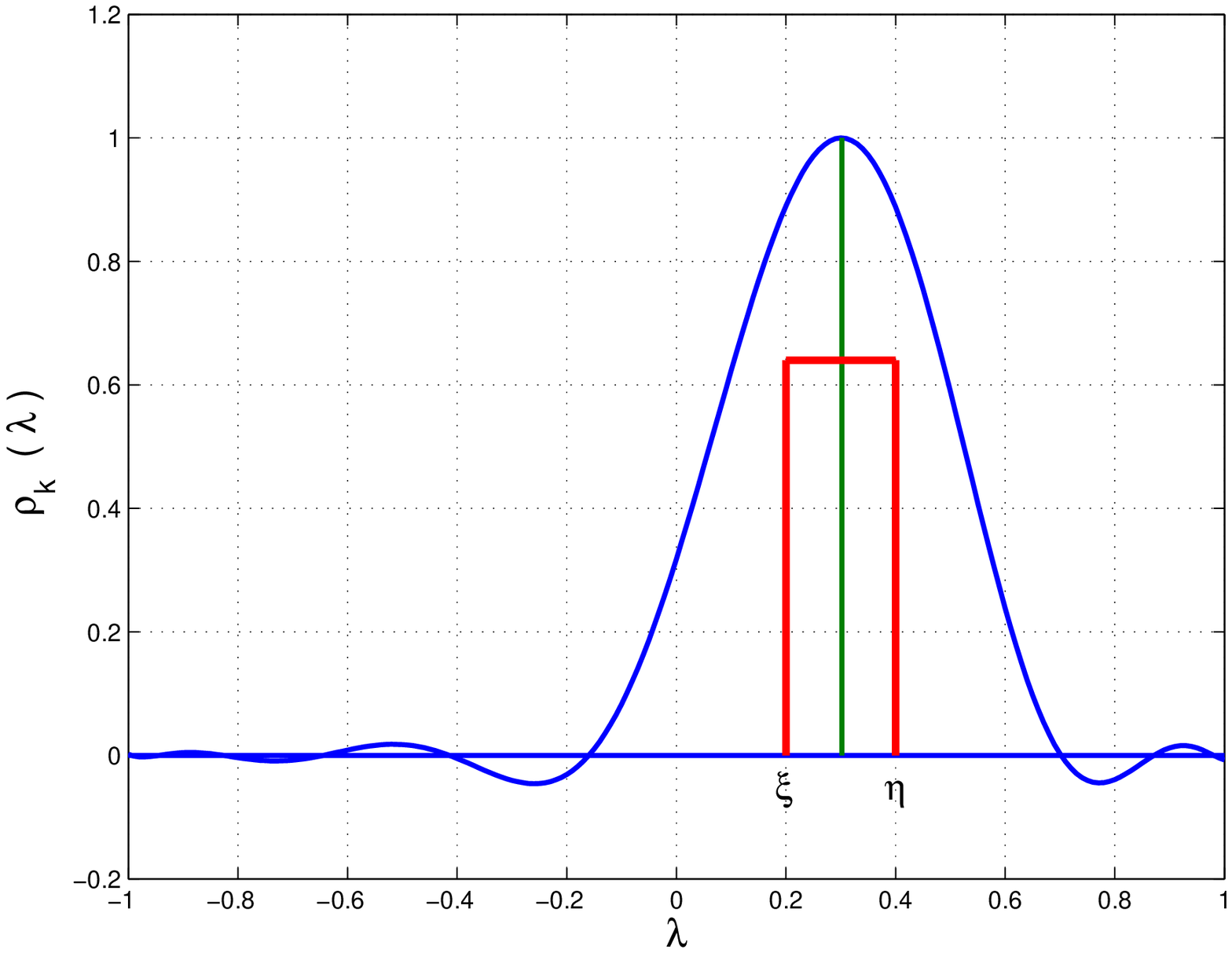}
\caption{Balancing the polynomial filter. Here the center
$\gamma$ of the Dirac delta  function 
is moved slightly to the right to yield a balanced 
polynomial.}\label{fig:onefilt}
\end{figure}

To adjust the center $\gamma$ so that $\rho_k(\xi) =  \rho_k(\eta)$, it is important to use the variable $\theta = \cos \inv t$ which plays a prominent role in the definition of
 Chebyshev polynomials.
We will denote by $\theta_x $ the angle $\theta_x = \cos^{-1}( x) $.
Thus $\cos(\theta_\xi ) = \xi$, $\cos(\theta_\eta ) = \eta$, 
$\cos(\theta_\gamma ) = \gamma$, etc.
We then apply Newton's method to solve the equation 
\eq{eq:Neq} 
 \rho_k (\xi) - \rho_k (\eta) = 0,
\en 
with respect to $\theta_\gamma$, through the coefficients $\mu_j$  
(see \nref{eq:polPJ} and \nref{eq:mu}).
The polynomial $\rho_k$ 
can be written in terms of the variable $\theta = \cos\inv (t) $ as
\[
\rho_{k} (\cos  \theta ) = \sum_{j=0}^k g_j^{k} \cos(j \theta_\gamma)
\cos(j\theta). \] Note that the first damping coefficient $g_0^{k}$
is multiplied by 1/2 to simplify notation, so that the first term with $j=0$ in the
expansion  is  not $1$
but  rather
$1/2$.   
In this way \nref{eq:Neq}    becomes    
\eq{eq:Neq2}   f(\theta_\gamma)    \equiv
\sum_{j=0}^k  g_j^{k}  \cos(j  \theta_\gamma)  [\cos(j\theta_\xi)-
  \cos(j\theta_\eta)] = 0  .  \en 

In order to solve this equation
by Newton's method, we need the derivative of $f$
with respect to $\theta_\gamma$ which is readily computable:
\[
f'(\theta_\gamma) = - 
\sum_{j=0}^k  g_j^{k}  j  \sin(j  \theta_\gamma)  [\cos(j\theta_\xi)-
  \cos(j\theta_\eta)]. \] 
Furthermore, as it turns out, 
the mid-angle   
\eq{eq:initGam} 
\theta_c = \frac{1}{2}(\theta_\xi + \theta_\eta)
\en 
already  provides a good initial guess for a balanced filter, though 
the values at the boundaries differ slightly as is shown in Figure~\ref{fig:onefilt}.  
In fact, the initial value $\theta_c$ in \nref{eq:initGam}  
is good enough to yield convergence of  the Newton iteration in 1 or 2
steps  in many  cases.  However,  there  may be  difficulties for  low
degree polynomials. Thus, if the  Newton's scheme fails to converge in
$2$ steps,  we compute the roots
of \nref{eq:Neq2} exactly via an eigenvalue problem, see Appendix \ref{appdix:1}.

\section{The Thick-Restart filtered Lanczos algorithm with deflation}\label{sec:filtlan} 
We now discuss how to efficiently combine a Thick-Restart (TR) version
of the  Lanczos algorithm with the polynomial  filtering and deflation
to compute all  the eigenpairs inside an arbitrary  interval. Since we
will apply  the Lanczos algorithm to  $\rho_k(\widehat{A})$ instead of
$A$, the following discussion will  deal with a given Hermitian matrix
denoted  by  $B$,  to  remind  the  reader  of  the  use  of  spectral
transformations  in  the   actual  computation.  A  filtered  subspace
iteration scheme is also discussed.
\subsection{The  Lanczos algorithm}
\label{sec:lan}
The Lanczos algorithm builds an orthonormal basis of the Krylov subspace
\[
\mathcal{K}_m = \mathrm{span} \{ q_1, B q_1, \ldots, B^{m-1} q_1 \}  \]
by a Gram-Schmidt process in which, at step $j$, the vector 
$Bq_j$  is  orthogonalized against $q_j$ and (when $j>1$) against $q_{j-1}$.
In the Lanczos algorithm the matrix $B$ is needed only 
 in the form of matrix-vector products, which may be very economical when
$B$ is sparse. In addition, no costly pre-processing is required as is the case for
codes based on shift-and-invert \cite{NASTRAN} or rational filtering \cite{Stefan,Feast,SakSug03,SakTad07}.

The sequence of vectors computed in  the course of the Lanczos algorithm
satisfies the 3-term recurrence: 
\begin{equation}
\label{eq:3term}
\beta_{i+1} q_{i+1} = B q_i - \alpha_i q_i - \beta_i q_{i-1}.
\end{equation}
Therefore, in principle
 only three Lanczos vectors need to be stored in main memory.
As is well-known, 
in \emph{exact}  arithmetic, this 3-term recurrence would  deliver 
an orthonormal basis $\left\{q_1,\dots,q_m\right\}$ of 
$\mathcal{K}_m$. 
In the presence of rounding, orthogonality between the $q_i$'s  
is quickly lost, and so
a form of reorthogonalization is needed in practice and this will be discussed
shortly.
The Lanczos procedure is sketched in Algorithm~\ref{alg:lanczos}, in which matrix 
$Q_j \equiv [q_1,\dots,q_j]$ contains the basis constructed up to step $j$ as its
column vectors.

\begin{algorithm}
\caption{The $m$-step Lanczos algorithm.}
\begin{algorithmic}[1]
\State Input:  a Hermitian matrix $B\in\IC^{n\times n}$, and
an initial unit vector $q_1\in\IC^n$.
\State $q_0:=0$, $\beta_1:=0$
\For{$i=1,2,\dots,m$}
\State $w := B q_i - \beta_i q_{i-1}$
\State $\alpha_i := q_i^H w$
\State $w := w - \alpha_i q_i$
\State Reorthogonalize: $w := w - Q_j (Q_j^H w)$ \Comment{classical Gram-Schmidt}
\State $\beta_{i+1} := \|w\|_2$
\If{$\beta_{i+1}=0$}
\State $q_{i+1}=$ a random vector of unit norm that is orthogonal to $q_1, 
\ldots q_i$
\Else
\State $q_{i+1} := w/\beta_{i+1}$
\EndIf
\EndFor
\end{algorithmic}
\label{alg:lanczos}
\end{algorithm}

Let $T_m$ denote the
symmetric tridiagonal matrix
\begin{equation}
\label{eq:Tm}
T_m  =  \left [ \begin{array}{cccc}
\alpha_1  & \beta_2  & &  \\
\beta_2   & \alpha_2 & \ddots &  \\
          & \ddots   & \ddots & \beta_{m}  \\
          &                        & \beta_{m}   & \alpha_m \\
\end{array} \right ],
\end{equation}
where  the scalars $\alpha_i,\beta_i$ are those produced by the
Lanczos algorithm.
Relation (\ref{eq:3term}) can be rewritten in the form: 
\begin{equation}
\label{eq:AQ=QT}
B Q_m = Q_m T_m + \beta_{m+1} q_{m+1}e_{m}^H ,
\end{equation}
where $e_m$ is the $m$th column of the canonical basis and
$q_{m+1}$ is the last vector computed by the $m$-step Lanczos algorithm.
Let $(\theta_i,y_i)$ be an eigenpair of $T_m$.
In case of ambiguity,
$(\theta_i^{(m)}, y_i^{(m)})$ will denote the same eigenpair
at the $m$th step of the process.  
Then the eigenvalues $\theta_i\Sup{m}$, known as Ritz values,
will approximate 
some of the eigenvalues of $B$ as $m$ increases.
The vectors $u_i^{(m)}=Q_m y_i\Sup{m}$, referred to as Ritz vectors, 
will approximate the related eigenvectors of $B$.
The Lanczos algorithm quickly yields good approximations
to extreme eigenvalues of $B$ while convergence is often 
much slower for those eigenvalues located deep inside
 the spectrum    \cite{Parlett-book,Saad-book3}.

As was already mentioned, 
the $q_i$'s  form an orthonormal basis in theory, but in practice they loose 
their orthogonality soon after at least one eigenvector starts converging,
leading to an unstable underlying computation.
This was studied in detail by Paige
in the 1970s \cite{paige:phd71,paige:lanczos72,paige:lanczos76}. 
A remedy to this problem is to reorthogonalize the vectors when needed.
Since we will use a restarted form of the Lanczos algorithm with
moderate dimensions, we decided to apply full re-orthogonalization 
to enforce orthogonality among the $q_i$'s to working precision 
(Line 7 in Algorithm~\ref{alg:lanczos}).

\subsubsection{Thick restart}\label{sec:TR}
The   method   \texttt{Filtlan}  described   in   \cite{Filtlan-paper}
essentially employs  the non-restarted  Lanczos algorithm with  $B$ in
the form of a  polynomial in $A$ and the $i$ loop  is not halted after
$m$  steps but  only  when  all eigenvalues  inside  the interval  are
captured. One main  issue with \texttt{Filtlan} is that  if the number
of eigenvalues  inside the interval  is large or the  eigenvalues near
the  boundaries of  the interval  are  clustered, then  the number  of
Lanczos steps  required by \filtlan\  may become quite large  and this
can limit  its applicability.  This is because  reorthogonalization is
necessary and becomes  expensive.  Memory  cost becomes another
issue as  large bases  must now  be saved.  An undesirable  feature of
non-restarted procedures  is that we  do not  know in advance  how big
these  bases can  be,  because the number of steps required 
 for convergence to take place is unknown.
Therefore, for very large problems for which limited memory procedures are required, restarting becomes mandatory.
In the algorithms described in this paper, we limit the total memory need to 
that of $m+1$ vectors plus an additional set of $nev$ vectors for the computed eigenvectors.


In a standard restarted procedure, an initial vector $q$ is selected 
from the current Lanczos  iterations and the Lanczos procedure
is restarted with this vector as the initial vector $q_1$.
We 
adopt the TR procedure~\cite{Andreas-al-thickr,wusi:00} as it
blends quite naturally with the filtering technique employed here.
The main idea of the TR procedures is to restart not with one vector
but with multiple ``wanted'' Ritz vectors. This technique implements essentially the idea of implicit restarting
\cite{leso:96} in a different form. 

Here we recall the main steps of TR, assuming
for simplicity that there are no locked vectors yet. In this case, suppose that after $m$ steps of the Lanczos algorithm, we end up with, say $l$ 
Ritz vectors $u_1, u_2, \ldots, u_l$ that we wish to use as the restarting vectors along with the last vector $q_{m+1}$. 
An important observation is that each Ritz vector has a residual in the direction of
$q_{m+1}$~\cite{Saad-book3}. Specifically, we have
\eq{eq:res} 
(B - \theta_i I) u_i = (\beta_{m+1} e_m^T y_i) q_{m+1} \equiv s_i q_{m+1}, 
\quad \mbox{for} \quad i=1,\ldots,l .
\en
Let $\hat q_i = u_i$ and $\hat q_{l+1}=q_{m+1}$. 
Rewriting \nref{eq:res} in a matrix form, it follows that
\eq{eq:tr2}
B \hat Q_l = \hat Q_l \Theta_l + \hat q_{l+1}s^H , 
\en
where $\hat Q_l = \left[\hat q_1, \ldots, \hat q_l \right]$, 
the diagonal matrix $\Theta_l$ consists of the corresponding Ritz values, 
and the vector $s^H$ is given by $s^H=\left[s_1,\ldots,s_l\right]$.
After restart, to perform one step from $\hat q_{l+1}$, we proceed as in the 
Arnoldi algorithm:
\eq{eq:trstep}
\beta_{l+2} \hat q_{l+2} = B \hat q_{l+1} - \sum_{i=1}^{l+1} (B \hat q_{l+1}, \hat q_i) \hat q_i  = B \hat q_{l+1} - \sum_{i=1}^{l} s_i \hat q_i - \alpha_{l+1} \hat q_{l+1},
\en
where $\alpha_{l+1}=(B \hat q_{l+1}, \hat q_{l+1})$.
Thus, after completing 
 the first step after the restart, we get the factorization
\eq{eq:TRstep}
B \hat Q_{l+1} = \hat Q_{l+1} \hat T_{l+1} + \beta_{l+2} \hat q_{l+2} e_{l+1}^H
\quad 
\mbox{with}
\quad
\hat T_{l+1} = \begin{pmatrix}
                \Theta_l & s \\
                s^H & \alpha_{l+1}
               \end{pmatrix}.
\en
From this step on, the algorithm proceeds in the same way as 
the standard Lanczos algorithm, i.e., $\hat q_{k}$ for $k\ge l+3$ is computed using the 3-term recurrence, until the dimension reaches $m$, and
the matrix $\hat T_m$ will be completed by the tridiagonal submatrix consisting of $\alpha_k$ and $\beta_k$.
An induction argument will show that 
a vector $q_j$ computed in this way will be orthogonal to all previous
$q_i$'s. 

\subsubsection{Deflation}
In addition to TR, an essential ingredient to the success of 
the procedure is the inclusion of `locking', i.e., explicit deflation. Once an eigenpair has converged, we add the eigenvector
as a column to a set $U$ of the computed eigenvectors and exclude it from the search subspace, i.e., we will compute the eigenpairs of  $(I-UU^H) B$ in the subsequent iterations. There are several benefits to the use of locking. First, all the locked eigenvalues will be transformed to zeros. This is critical because when deflation is used, an eigenvector cannot be computed a second time. Second,
deflation helps to capture any eigenpair that is still left in the interval. In essence this procedure acts as a form of preconditioning by creating larger gaps among  the eigenpairs not yet discovered, resulting in faster convergence.  Finally, locking provides an efficient means to compute multiple or clustered eigenvalues without resorting to a block method. Since the filter $\rho_k$ may transform some distinct eigenvalues inside $[\xi, \ \eta]$ into multiple ones, locking becomes necessary to prevent missing eigenvalues.
  
\subsection{Practical details}\label{sec:alg} 
\label{sec:all}
We now put together the three important
ingredients  of polynomial  filtering, restarting, and deflation. 
In  its simplest form, the filtered Lanczos procedure
will first perform $m$  steps  of the  Lanczos
algorithm with the  matrix $B = \rho_k(\widehat{A})$, and  then it will restart.
Once the $m$ Lanczos steps are executed we obtain $m$ Ritz values $\theta_1, \theta_2, \ldots, \theta_m$, such that
\[
 \theta_1 \ge  \theta_2 \ge \cdots   \ge \theta_{l} \ge \rho_k({\xi}) > \cdots \ge \theta_m .
\]
These are eigenvalues of the tridiagonal matrix $T_m$ in \nref{eq:Tm}, with
$y_1, y_2, \ldots, y_m$ being the associated eigenvectors. 
The approximate eigenvectors of $B$ are the Ritz vectors $u_j = Q_m y_j$. 
Recall from  Section~\ref{sec:filters} that when the polynomial filter is constructed, 
we select a 
``\emph{bar}'' value~$\phi$, equal to $\rho_k(\xi)$ and $\rho_k(\eta)$, 
that separates  the \emph{wanted}  eigenvalues  (those in
\intv)  from  \emph{unwanted}  ones. 
 The eigenvalues $\theta_i$   below  
$\phi$, i.e., $\theta_{l+1},\ldots,\theta_{m}$, 
  are discarded.   For $\theta_{1},\ldots,\theta_{l}$, we  compute  the 
associated
(unit norm) Ritz vectors $ u_1,  u_2, \ldots, u_l$
and evaluate  
their  Rayleigh  quotients relative to $A$, which is
$\tilde \lambda_i =   u_i^H A  u_i$.   A second check is
performed at this point that discards any $\tilde \lambda_i$ falls outside \intv. 

For the remaining  approximate eigenpairs $(\tla_i,
u_i)$  we compute  the residual  norms $\|  A   u_i-
\tilde \lambda_i   u_i \|_2$.  
If a Ritz pair  has converged, we add the Ritz vector to the ``locked set'' of vectors.  
All future  iterations will
perform  a  deflation step  against  such a set.
The other, unconverged, candidate  eigenvector approximations are added to the ``TR set'', and the algorithm is then restarted.
Thus, there are three types of basis vectors used at any given step:
the locked vectors (converged eigenvectors), the vectors selected for TR, and the remaining Lanczos vectors.
The whole procedure is sketched in Algorithm~\ref{alg:Filantres}. 

\begin{algorithm}
\caption{The Filtered Lanczos algorithm with TR and deflation.\label{alg:Filantres}}
\begin{algorithmic}[1]
\State Input:  a Hermitian matrix $A\in\IC^{n\times n}$, and
an initial unit vector $q_1\in\IC^n$.

\State Obtain polynomial $\rho_k(t)$ and ``bar'' value $\phi$  

\State $q_0:=0$, $\beta_1:=0$, $Its:=0$, $lock := 0$, $l:=0$, $U := [\ ]$

\While{$Its \leq MaxIts$} 

\State If $l>0$, perform a TR step \nref{eq:trstep}, which results $\hat Q_{l+2}$ and $\hat T_{l+1}$
in \nref{eq:TRstep} 
\For{$i=l+1,\dots,m$}
\State Perform Lines~4 to 13 of 
Algorithm \ref{alg:lanczos} with $B = (I-UU^H)\rho_k(\widehat{A})$
\State Set $Its :=Its+1$
\EndFor
\State Result: $\hat Q_{m+1}  \ \in \ \IC^{n \times (m+1)}$, 
$ \hat T_m \ \in \ \IR^{m\times m}$
\State Compute candidate Ritz pairs, i.e.,
$(\theta_j, u_j)$ with $\theta_j\ge \phi$
\State Set $\hat Q := [\ ]$ and $l:=0$
\For{each candidate pair $(\theta_j, u_j)$}
\State Compute $\tilde \lambda_j =  u_j ^H A  u_j$  
\State If $\tilde \lambda_j \notin [\xi,\ \eta]$ ignore this pair 
\If{\{$(\tilde \lambda_j, u_j)$ has converged\}}
\State Add $u_j$ to Locked set $U := [U,  u_j]$
\State Set $lock:=lock+1$
\Else
\State Add $u_j$ to TR set $\hat Q := [\hat Q, u_j]$
\State Set $l:=l+1$
\EndIf  
\EndFor
\If{\{No candidates found\}  or \{No vectors in TR set\}} 
\State Stop
\EndIf
\EndWhile
\end{algorithmic}
\end{algorithm}

The following notes may help clarify certain aspects of the algorithm. 
Lines 6-9 execute a cycle of the Lanczos procedure that
performs matrix-vector products with $\rho_k(\widehat{A})$ and orthonormalizes
the resulting vectors against the vectors in the locked set.
In Line 11, eigenpairs of the tridiagonal matrix $T_m$ are computed and 
the Ritz values 
below the threshold $\phi $ are rejected.
The algorithm then computes 
the Ritz vectors associated with the remaining Ritz values.
The loop starting in Line 13  determines if these are to be rejected
(Line 15), put into the Locked set (Lines 16-18),
or into the TR set (Lines 20-21). 

To reduce computational costs, it is preferable to restart whenever enough 
eigenvectors have converged for $B$.
In Algorithm~\ref{alg:Filantres}, the Lanczos process is restarted when the dimension reaches $m$, i.e., when $i=m$.
A check for early restart is also triggered whenever every $Ncycle$ steps have been performed and $Ntest$ steps have been executed since the last restart, i.e., when $Its \mbox{ mod } Ncycle = 0$ and $i\ge l+ Ntest$.
In our implementation, $Ntest=50$ and $Ncycle=30$ are used.
In this check, an  eigendecomposition of $T_{i}$ 
is performed  and the number of eigenvalues larger than $\phi$,
that have  converged,  is counted.
If this number is larger  than $nev-lock$, where $nev$ is an estimate of the number of eigenvalues inside the concerned interval  given by the DOS algorithm,
the algorithm  breaks   the  $i$   loop  and restarts.

Another implementation detail is that the orthogonalizations
in Line~7  of Algorithm~\ref{alg:lanczos} and 
Line~7 of Algorithm~\ref{alg:Filantres} consist of at most two steps of the 
classical Gram-Schmidt orthogonalization \cite{luc05,Giraud20051069} with the DGKS \cite{citeulike:7035979} correction.

\subsection{Filtered subspace iteration}\label{sec:filtsub}
The described filtering  procedure can also be  combined with subspace
iteration~\cite{Saad-book3},  which  may  be advantageous  in  certain
applications,  such  as in  electronic  structure calculations based on 
 Density Function Theory (DFT) ~\cite{escm09}. This section gives a brief description 
 of this approach to point out the main differences with the TR
Lanczos. 

Given a matrix $B$ and block $X$ of $s$ approximate eigenvectors, 
the subspace iteration is of the form $\bar X \gets B X$, where 
the columns of $\bar X$ are then transformed via the Rayleigh--Ritz step~into Ritz vectors 
that are used as an input for the next iteration. 
The process is repeated until convergence is reached, 
and is guaranteed to deliver eigenvectors corresponding to the $s$ 
dominant eigenvalues of $B$, i.e., those with largest magnitude.
%
The method is attractive for its remarkable simplicity, robustness, and 
low memory requirement \cite{saad15}. 
In contrast to the Lanczos algorithm, it provides more flexibility in enabling 
the use of a good 
initial guess, which can lead to a significant reduction of computations in some
situations. 
Furthermore, as a block method, the subspace iteration offers an 
additional level of concurrency 
and allows for a more intensive use of BLAS3 operations.

When combined with filtering, the subspace iteration is applied to the
matrix $B  = \rho_k(\widehat{A})$,  whose largest eigenvalues  are the
images of the wanted eigenvalues of  $A$ from $[\xi,\ \eta]$ under the
approximate delta  function transform  $\rho_k(t)$.  
%
This behavior is  different from that of the  filtered Lanczos method,
which also  converges to the bottom  part of the spectrum  of $B$, and
therefore requires  detection of  converged smallest  eigenvalues as is
done  in step  11 of Algorithm~\ref{alg:Filantres} that selects
the Ritz values above the threshold $\phi$.

Note that,  subspace iteration requires  a reasonable estimate  $s$ of
the number  of eigenvalues of $A$  in the interval $[\xi,\  \eta]$. As
has already been discussed, such an estimate can be obtained, from the
DOS  algorithm  \cite{LinYangSaad2013-TR},  and  should  ideally  only
slightly  exceed the  actual number  of wanted  eigenvalues. Moreover,
subspace iteration  can also benefit  from deflation of  the converged
eigenpairs \cite{Saad-book3,ls_feast}.

\section{Numerical experiments}\label{sec:num}
\texttt{FiltLanTR}  has been  implemented  in  C and  the
experiments  in  Sections \ref{subsec:lap3d}-\ref{subsec:parsec}  were
performed  in sequential  mode  on  a Linux  machine  with Intel  Core
i7-4770 processor  and 16G memory.  The OpenMP experiments  in Section
\ref{subsec:openmp}  were  performed on  a  node  of Cori  (Cray  XC40
supercomputer,  Phase~I) at  the National  Energy Research  Scientific
Computing Center (NERSC).   Each node of this system  has two sockets,
with each socket  populated with a 16-core  Intel ``Haswell'' processor
at 2.3 GHz (i.e.,  32 cores per node), and is equipped  with 128 GB of
DDR4 2133Mhz MHz  memory. The code was compiled with  the gcc compiler
using the -O2  optimization level.  The convergence  tolerance for the
residual norm  was set at  $10^{-8}$ and the  Lanczos $\sigma$-damping
was used for all polynomial filters.


\subsection{3D Laplacian}\label{subsec:lap3d}
We first use  a model problem to illustrate the  effects of the ``bar'' value $\phi$ and
the   number   of   slices     on   the   performance   of
\texttt{FiltLanTR}.  The  model problem  was  selected  as a  negative
Laplacian  $-\Delta$ operator  subject to  the homogeneous  Dirichlet
boundary conditions over the unit  cube $[0,\ 1]^3$. Discretizing this
operator on a  $60\times 60\times 60$ grid with  the $7$-point stencil
results  in a matrix of size $n=216,000$. The  spectrum of  this
matrix is inside the interval  $[0.00795,\ 11.99205]$. The goal is to  compute
all the $3406$ eigenpairs inside the interval $[0.6,\ 1.2]$.

\paragraph{Selection of $\phi$} 
In the first experiment we fixed the number of slices at $10$ and varied the value of $\phi$ to study its influence on \texttt{FiltLanTR}.
The DOS algorithm \cite{LinYangSaad2013-TR}  was exploited to 
partition $[0.6,\ 1.2]$ into $10$ slices $[\xi_i,\ \eta_i]_
{i=1,\ldots,10}$  so that each $[\xi_i,\ \eta_i]$ contains 
roughly $340$ eigenvalues, as shown in the second column of 
Table \ref{tab:Laplacian}. To validate the effectiveness of 
this partitioning, the exact number of eigenvalues in each 
$[\xi_i,\ \eta_i]$ is provided in the third column of the same table. 
As can be seen, the difference between the approximate and the exact 
number is less than $10$ for most sub-intervals. 
In the  following discussion, we will 
 refer to each sub-interval by its index $i$ shown in 
the first column of Table \ref{tab:Laplacian}. 
\begin{table}[tbh]
\caption{Partitioning $[0.6,1.2]$ into $10$ sub-intervals for the 3D discrete Laplacian example.}
\label{tab:Laplacian}
\centering \tabcolsep8pt
\begin{tabular}
[c]{c|c|c|c}\hline
$i$ &  $[\xi_i, \eta_i]$ & $\eta_i-\xi_i$ & \#eigs  \\\hline
1&$[0.60000, 0.67568]$ & 0.07568 & 337   \\ 
2&$[0.67568, 0.74715]$ & 0.07147 & 351  \\ 
3&$[0.74715, 0.81321]$ & 0.06606 & 355  \\ 
4&$[0.81321, 0.87568]$ & 0.06247 & 321   \\ 
5&$[0.87568, 0.93574]$ & 0.06006 & 333    \\ 
6&$[0.93574, 0.99339]$ & 0.05765 & 340  \\ 
7&$[0.99339, 1.04805]$ & 0.05466 & 348   \\ 
8&$[1.04805, 1.10090]$ & 0.05285 & 339  \\ 
9&$[1.10090, 1.15255]$ & 0.05165 & 334   \\ 
10&$[1.15255, 1.20000]$& 0.04745 & 348  \\ \hline
\end{tabular}
\end{table}

The computations of the eigenpairs were 
performed independently for each sub-interval. 
The Krylov subspace dimension and the maximum
iteration number were set as $4nev$ and $16nev$, respectively, where
$nev$ is the estimated number of eigenvalues in each sub-interval. In
this example we have $nev=341$.

 Tables~\ref{tab:Laplacianresult}-\ref{tab:Laplacianresult3}, show
results of the \texttt{FiltLanTR} runs for which the values of $\phi$ have 
been  set to $0.6$, $0.8$, and $0.9$, respectively.  The tables list
the degree of the filter  polynomials, the number of total iterations,
number of  matrix-vector products (matvecs), the CPU  time for matrix-vector
products,  the total  CPU  time, as  well as  the  maximum and  average
residual           norms of the computed eigenpairs.
The following observations can be made.  
With a fixed $\phi$, the iteration number
is  almost  the  same  for  each  $[\xi_i,\  \eta_i]$  (especially  in
Tables~\ref{tab:Laplacianresult} and \ref{tab:Laplacianresult3}).
However, 
the filter degree  grows for sub-intervals deeper inside the spectrum.
This results in a higher computational  cost per  iteration,  
as indicated by
the increase  in the number  of matvecs and the CPU time in columns five 
and six.
Note that the computational accuracy remains 
identical for different sub-intervals.

\begin{table}[tbh]
\caption{Numerical results for the 3D discrete Laplacian example with $\phi = 0.6$. The number of eigenvalues found inside each $[\xi_i, \eta_i]$ is equal to the exact number shown in Table \ref{tab:Laplacian}.}%
\label{tab:Laplacianresult}
\centering\tabcolsep4pt
\begin{tabular}
[c]{c|c|c|c|c|c|c|c}\hline
\multirow{2}{*}{Interval}  & \multirow{2}{*}{deg} & \multirow{2}{*}{iter} & \multirow{2}{*}{matvecs}& \multicolumn{2}{|c|}{CPU Time (sec.)}  & \multicolumn{2}{|c}{Residual}\\\cline{5-8}
 & & & & Matvecs & Total & Max & Avg \\\cline{1-8}
1 & 172 & 1567  & 270055  & 551.36 & 808.33 &$3.30\!\times\! 10^{-9}$  & $2.90\!\times\! 10^{-11}$\\ 
2 & 192 & 1514 &291246 & 594.83 & 840.29& $4.70\!\times\! 10^{-9}$ &$4.79\!\times \!10^{-11}$ \\ 
3 & 216  & 1513 & 327395 &671.19 & 918.16&$3.30\!\times\! 10^{-9}$ & $3.20\!\times \!10^{-11}$\\ 
4 &   237  & 1516 & 359863 & 736.95&973.66&$7.80\!\times\! 10^{-9}$ & $7.08\!\times \!10^{-11}$\\ 
5& 254 & 1543  & 392525 & 803.71&1051.91& $8.50\!\times\! 10^{-11}$&$1.20\!\times\! 10^{-12}$\\ 
6 & 272 &1511  &  411622 &843.84& 1084.82& $4.90\!\times \!10^{-9}$& $5.57\!\times\! 10^{-11}$ \\ 
7&  294   & 1562   & 459897&943.33&1202.39& $4.50\!\times\!10^{-9} $& $3.61\!\times \!10^{-11}$\\ 
8& 311 & 1511 &470589 &965.23&1205.87& $1.50\!\times\! 10^{-9}$ & $2.64\!\times\! 10^{-11}$\\ 
9&325    & 1565 & 509308  & 1045.54 & 1300.20& $1.10\!\times\! 10^{-9}$&  $1.73\!\times \!10^{-11}$\\ 
10& 361  &  1538 & 555948 &1140.49 &1392.46&$3.80\!\times\! 10^{-9}$ & $3.76\!\times\! 10^{-11}$\\\hline
\end{tabular}
\end{table}

\begin{table}[tbh]
\caption{Numerical results for the 3D discrete Laplacian example with $\phi = 0.8$. The number of eigenvalues found inside each $[\xi_i, \eta_i]$ is equal to the exact number shown in Table \ref{tab:Laplacian}.}%
\label{tab:Laplacianresult2}
\centering\tabcolsep4pt
\begin{tabular}
[c]{c|c|c|c|c|c|c|c}\hline
\multirow{2}{*}{Interval}  & \multirow{2}{*}{deg} & \multirow{2}{*}{iter} & \multirow{2}{*}{matvecs}& \multicolumn{2}{|c|}{CPU Time(sec.)}  & \multicolumn{2}{|c}{Residual}\\\cline{5-8}
 & & & & Matvecs & Total & Max & Avg \\\cline{1-8}
1 & 116 & 1814  & 210892  & 430.11 & 759.24 & $6.90\!\times \!10^{-9}$&$7.02\!\times\! 10^{-11}$\\ 
2& 129 & 2233 & 288681 & 587.14& 986.67& $5.30\!\times\! 10^{-9}$&$7.39\!\times\! 10^{-11}$\\ 
3 & 145  & 2225 & 323293 &658.44 & 1059.57& $6.60\!\times\! 10^{-9}$&$5.25\!\times\! 10^{-11}$\\ 
4 &   159  & 1785 & 284309 & 580.09&891.46& $3.60\!\times\! 10^{-9}$&$4.72\!\times \!10^{-11}$\\ 
5& 171 & 2239  & 383553 & 787.00&1180.67& $6.80\!\times\! 10^{-9}$&$9.45\!\times \!10^{-11}$\\ 
6 & 183 & 2262  &  414668 &848.71& 1255.92& $9.90\!\times\! 10^{-9}$&$1.13\!\times\! 10^{-11}$\\ 
7&  198   & 2277   & 451621&922.64&1338.47& $2.30\!\times\! 10^{-9}$&$3.64\!\times\! 10^{-11}$\\ 
8& 209 & 1783  &373211 &762.39&1079.30& $8.50\!\times\! 10^{-9}$&$1.34\!\times\! 10^{-10}$\\ 
9 & 219    & 2283 & 500774  & 1023.24 & 1433.04& $4.30\!\times\! 10^{-9}$&$4.41\!\times \!10^{-11}$\\ 
10 & 243  &  1753 & 426586 & 874.11 & 1184.76& $5.70\!\times\! 10^{-9}$&$1.41\!\times\! 10^{-11}$\\\hline
\end{tabular}
\end{table}

\begin{table}[tbh]
\caption{Numerical results for the 3D discrete Laplacian example with $\phi = 0.9$. The number of eigenvalues found inside each $[\xi_i, \eta_i]$ is 
equal to the exact number shown in Table \ref{tab:Laplacian}.}%
\label{tab:Laplacianresult3}
\centering\tabcolsep4pt
\begin{tabular}
[c]{c|c|c|c|c|c|c|c}\hline
\multirow{2}{*}{Interval}  & \multirow{2}{*}{deg} & \multirow{2}{*}{iter} & \multirow{2}{*}{matvecs}& \multicolumn{2}{|c|}{CPU Time (sec.)}  & \multicolumn{2}{|c}{Residual}\\\cline{5-8}
 & & & & Matvecs & Total & Max & Avg \\\cline{1-8}
1 & 80& 2636 & 211409  & 436.11 & 1006.91 & $9.90\!\times \!10^{-9}$&$1.90\!\times\! 10^{-10}$\\
2& 89 & 2549 & 227419 & 468.14& 990.79& $9.30\!\times\! 10^{-9}$&$1.62\!\times\! 10^{-10}$\\
3 & 100  & 2581 & 258682 &533.39 & 1076.35& $8.80\!\times\! 10^{-9}$&$1.80\!\times\! 10^{-10}$\\
4 &   110 & 2614 & 288105 & 594.18&1140.08& $9.90\!\times\! 10^{-9}$&$1.50\!\times \!10^{-10}$\\
5& 118 & 2606  &  308109& 635.90&1185.56& $9.80\!\times\! 10^{-9}$&$2.17\!\times \!10^{-10}$\\
6 & 126 & 2613  &  329855 &681.94& 1235.23& $6.40\!\times\! 10^{-9}$&$1.18\!\times\! 10^{-10}$\\
7&  137   & 2629   & 360834&746.59&1311.34& $8.10\!\times\! 10^{-9}$&$9.06\!\times\! 10^{-11}$\\
8& 145 & 2603  &378099 &781.48&1324.18& $5.00\!\times\! 
10^{-9}$&$3.45\!\times\! 10^{-11}$\\
9 &151   & 2581 & 390394  & 806.61 & 1339.95& $1.00\!\times\! 10^{-8}$&$3.06\!\times \!10^{-10}$\\
10 & 168  &  2575 & 433317 & 894.58 & 1428.38& $8.30\!\times\! 10^{-9}$&$2.84\!\times\! 10^{-10}$\\\hline
\end{tabular}
\end{table}

We next study  how  the statistics  in
Tables~\ref{tab:Laplacianresult}-\ref{tab:Laplacianresult3} change 
as  $\phi$ increases.  We illustrate this change through barcharts  in
Figure \ref{fig:Deg}. 
The first plot in Figure \ref{fig:Deg} shows that the degree of the
filter  decreases  monotonically  as  $\phi$ increases  for  all  $10$
intervals.  This   is  expected  from   the  way  the   polynomial  is
selected;     see     Section     \ref{sec:deltFun}     and     Figure
\ref{fig:findDeg}.  The  second subfigure  in the  same row
indicates  that fewer  iterations are  needed if  a smaller  $\phi$ is
specified. Thus, if available memory is tight one can trade memory for
computation by setting $\phi$ to a lower value, which would lead to
high degree polynomials and fewer iterations of \texttt{FiltLanTR}.
These first two subfigures 
lead to  the  conclusion that the total number of iterations
and the   polynomial degree change in opposite directions as  $\phi$
changes. 
It is hard  to predict  the value of  their product,
which is  the total  number of matvecs.

However,  the first
plot on the second row of Figure \ref{fig:Deg} indicates that a 
smaller $\phi$ can 
lead to a larger number of matvecs. 
Subfigure (2,2)  suggests
using  a moderate $\phi$  to reach  an optimal performance  over all
sub-intervals in terms of the  iteration time. 
The  default value of  $\phi$ is set  to 
$0.8$ in the subsequent experiments.


\begin{figure}[htb] 
\begin{tabular}
[c]{cc}%
\includegraphics[width=0.46\textwidth]{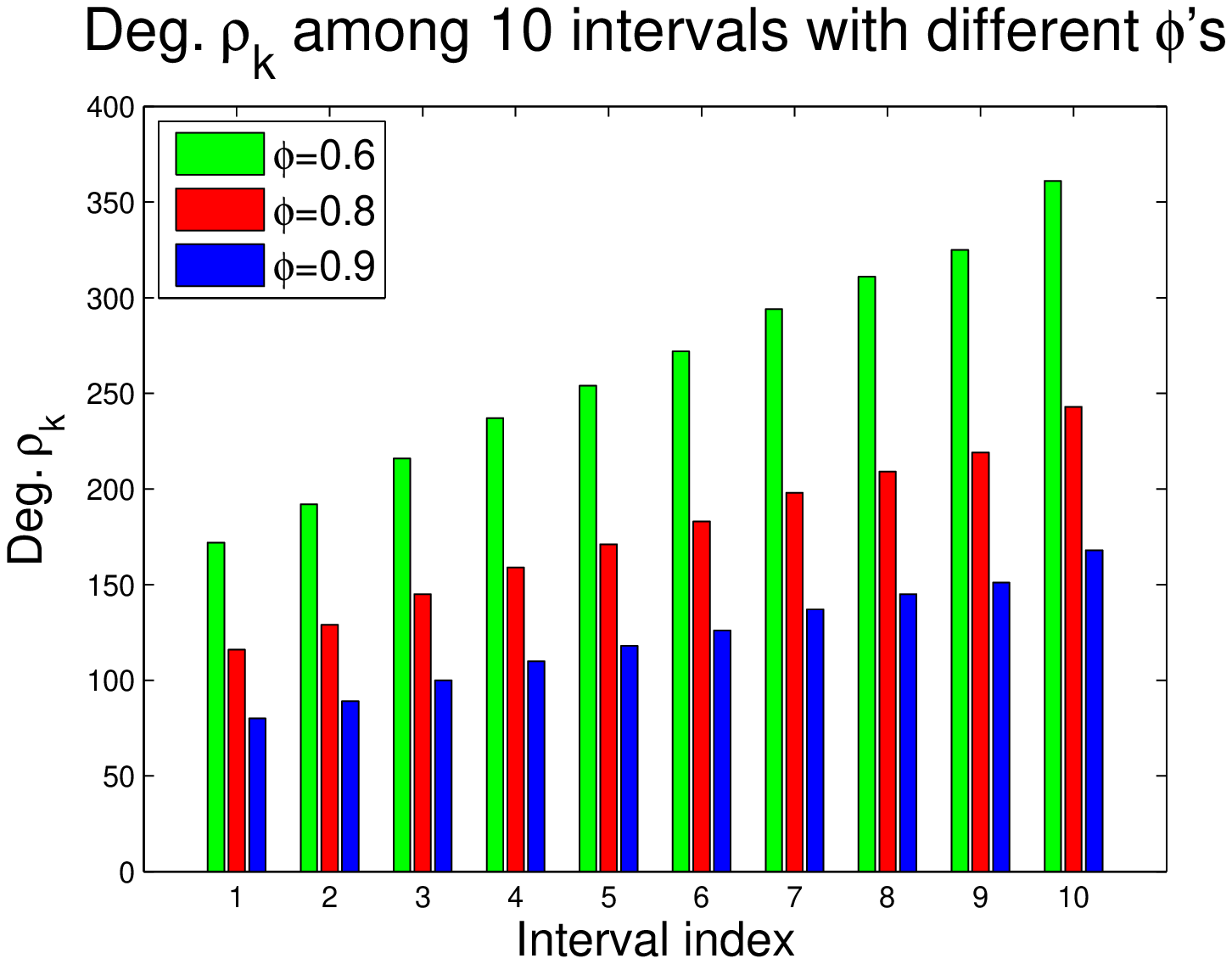}&
\includegraphics[width=0.46\textwidth]{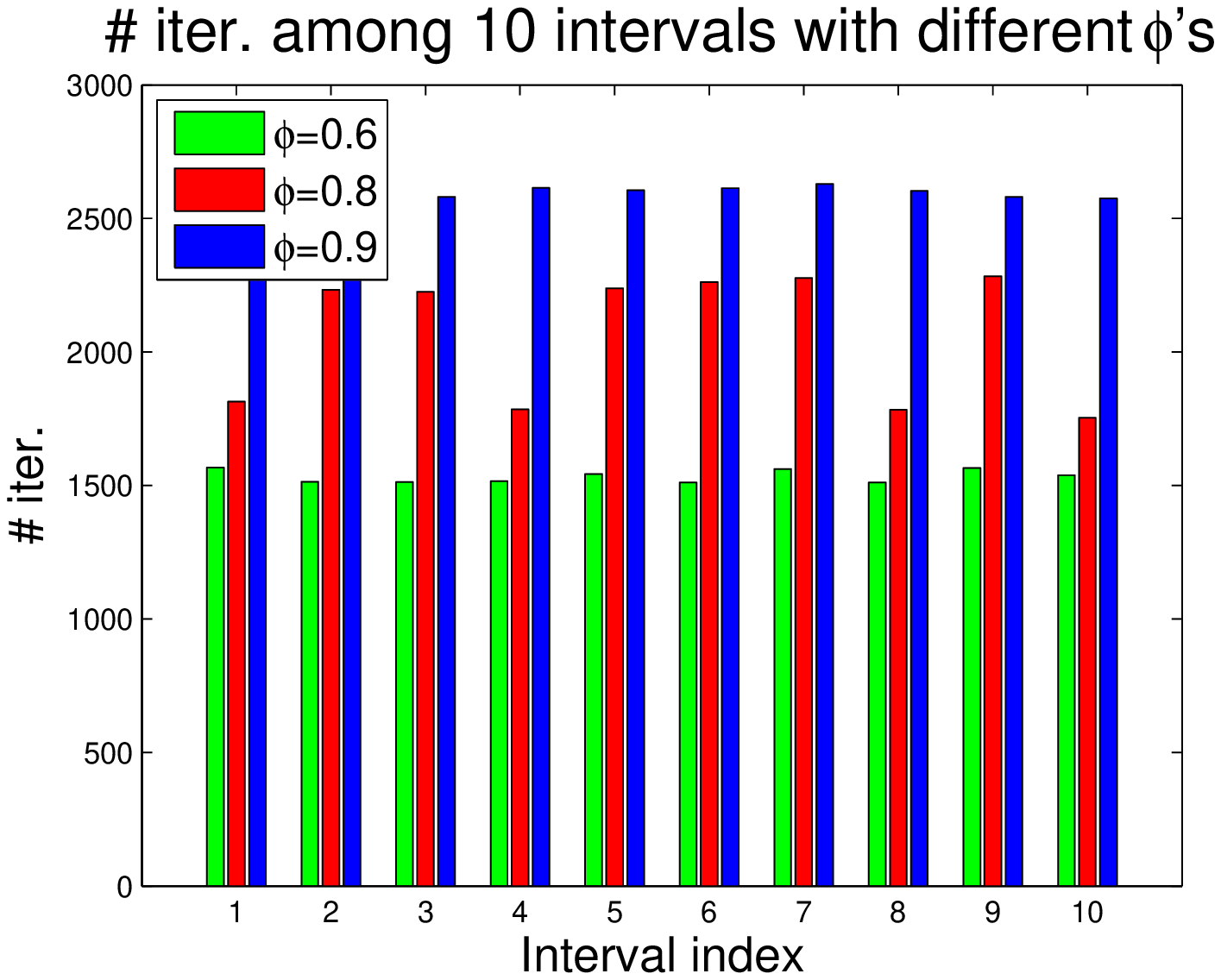} \\
\includegraphics[width=0.46\textwidth]{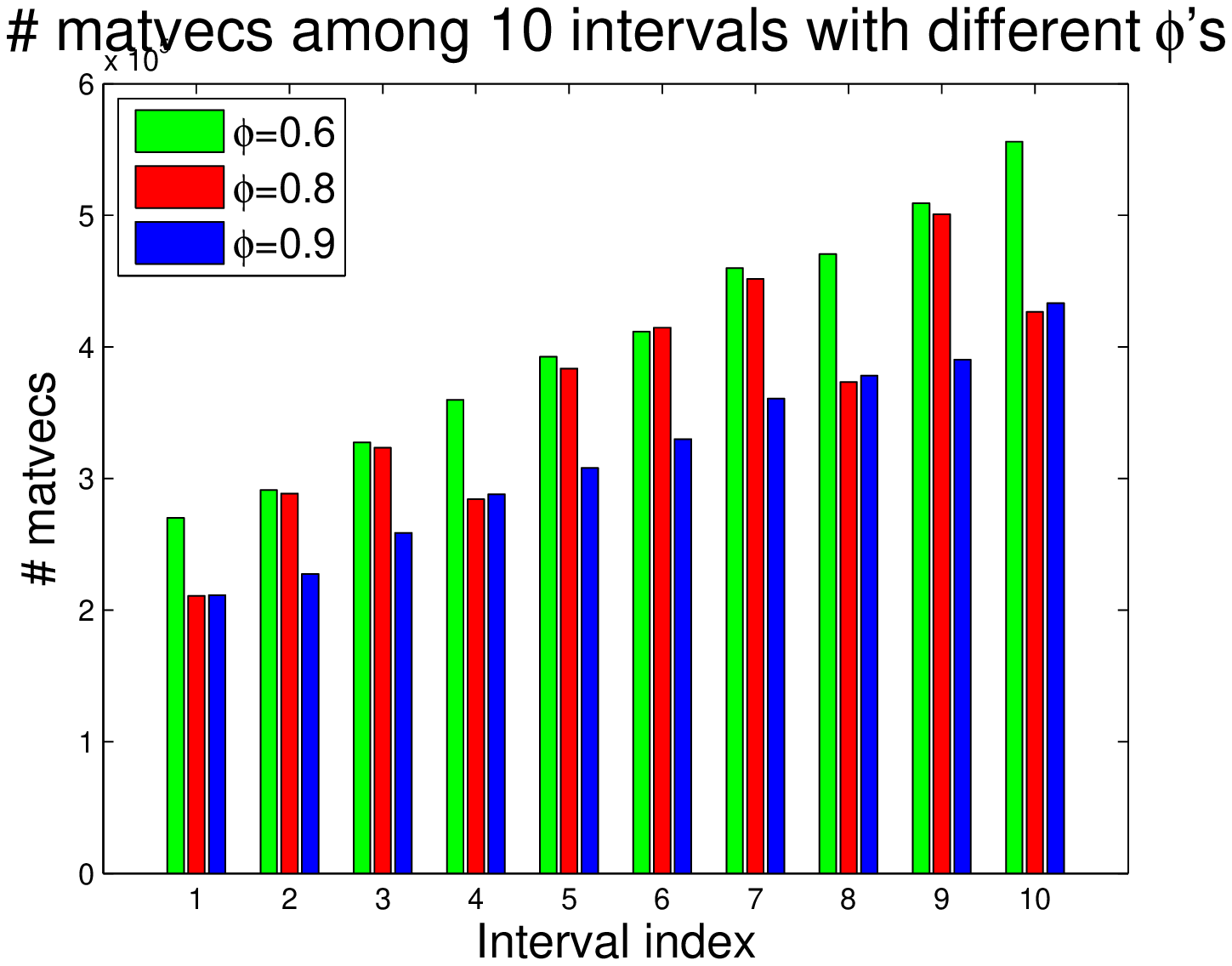}&
\includegraphics[width=0.46\textwidth]{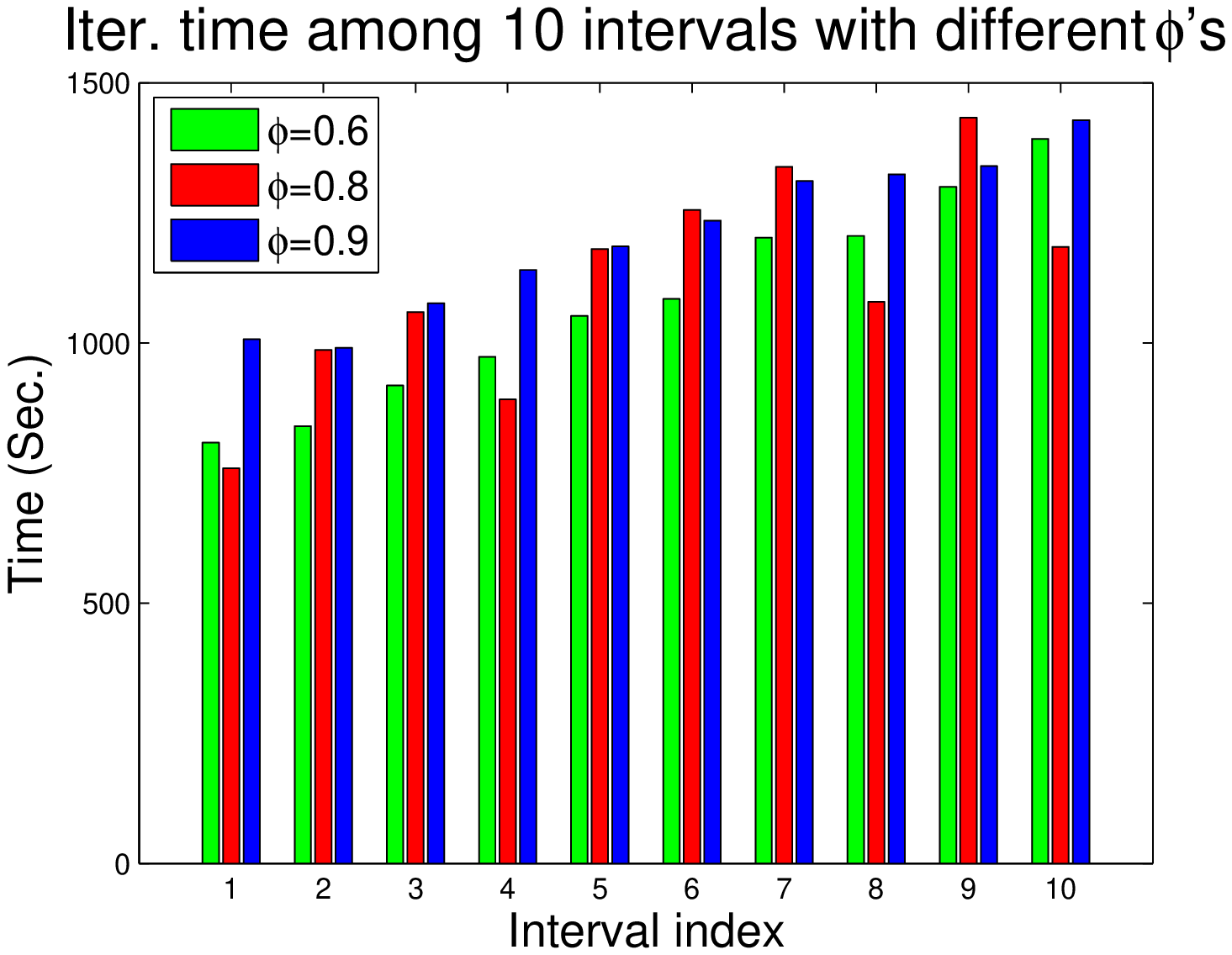}
\end{tabular}
\caption{Comparison of four statistics in Tables~\ref{tab:Laplacianresult}-
\ref{tab:Laplacianresult3} as $\phi$ increases from $0.6$ to $0.9$.}
\label{fig:Deg}
\end{figure}

\paragraph{Selection of the number of slices}
In the second experiment we fixed the value of $\phi$ at $0.8$ and varied the number of slices. 
The same $3$D discrete Laplacian eigenvalue problem was tested with $2$, $5$, $15$ and $20$ slices. 
The detailed computational results are tabulated in Appendix \ref{sec:differentslice}. 
For each value of the number of slices,  we averaged the statistics 
from this table across the  different slices and show the results in  
Table \ref{tab:Laplacianresultavg}.
From the  second column of Table  \ref{tab:Laplacianresultavg}, we see
that  on  average  the  degree of  the  polynomial  filters  increases
(roughly) linearly  with  the number  of  slices.  This observation  is  
further  supported by  the first  plot in  Figure \ref{fig:slice}.  

\begin{table}[tbh]
\caption{Averaged statistics across slices for different slice numbers,
 for the 3D discrete Laplacian example with $\phi = 0.8$.}%
\label{tab:Laplacianresultavg}
\centering\tabcolsep7.1pt
\begin{tabular}
[c]{c|c|c|c|c|c|c}\hline
\multirow{2}{*}{\#slices}  & \multirow{2}{*}{ deg} & \multirow{2}{*}{ iter} &\multicolumn{2}{|c|}{Matvecs}& \multicolumn{2}{|c}{CPU Time (sec.)} \\\cline{4-7}
 & & & $\#$/slice &  Time/slice & Total/slice &Total\\\cline{1-7}
2 & 34.5 & 9284.5 & 328832.0  & 681.74 & 11817.35 &23634.69\\
5 & 88.0  & 3891.8 & 347704.6 & 715.98 & 2126.97  &10634.85\\
10 & 177.2 & 2065.4 &  365758.8  & 747.69 &   1116.91 &11169.13\\
15 &266.1 &1351.9  &361809.0 &746.04 & 911.54   &  13673.12\\
20 & 356.8 & 1081.7 &  392083.3 & 807.46  &   909.62 & 18192.45\\\hline
\end{tabular}
\end{table}

If the number of iterations  per slice decreases linearly with  an
increasing number of
slices, then a constant matvec number and computational time per slice
would  be  expected   in  the  fourth  and  fifth   columns  of  Table
\ref{tab:Laplacianresultavg}. However, this  is not the case  due to a
slower reduction of the iteration number, as shown in the second plot
of the first row of Figure \ref{fig:slice}. 
The total iteration time essentially consists of 
two parts: the matrix-vector  product time and the reorthogonalization
time.   With too  few  slices,  it is reorthogonalization
that dominates the  overall computational cost. By contrast, too
many slices may lead to excessively high degree polynomial filters for
each slice, thus increase the  matvec time. This is  illustrated in the left
subfigure of  Figure \ref{fig:slice}  on the  second row.  

Finally, 
when the number of slices doubles  for a fixed interval, then, ideally, 
one would expect that the number of
eigenpairs in each slice is almost  halved. 
Then we would hope to
reduce the computational cost per slice roughly by half each time the number 
of slices is doubled, but
this is not true in practice. As shown in the last subfigure of the second row
of Figure \ref{fig:slice},
 by using more slices,  the gain
in computational  time per slice   will be
eventually  offset by the higher  degree of the  polynomial  filters. 
We found that a general rule of thumb 
is to have  roughly $200$  to $300$  eigenvalues per
slice.

\begin{figure}[htb] 
\centering
\begin{tabular}
[c]{cc}%
\includegraphics[width=0.46\textwidth]{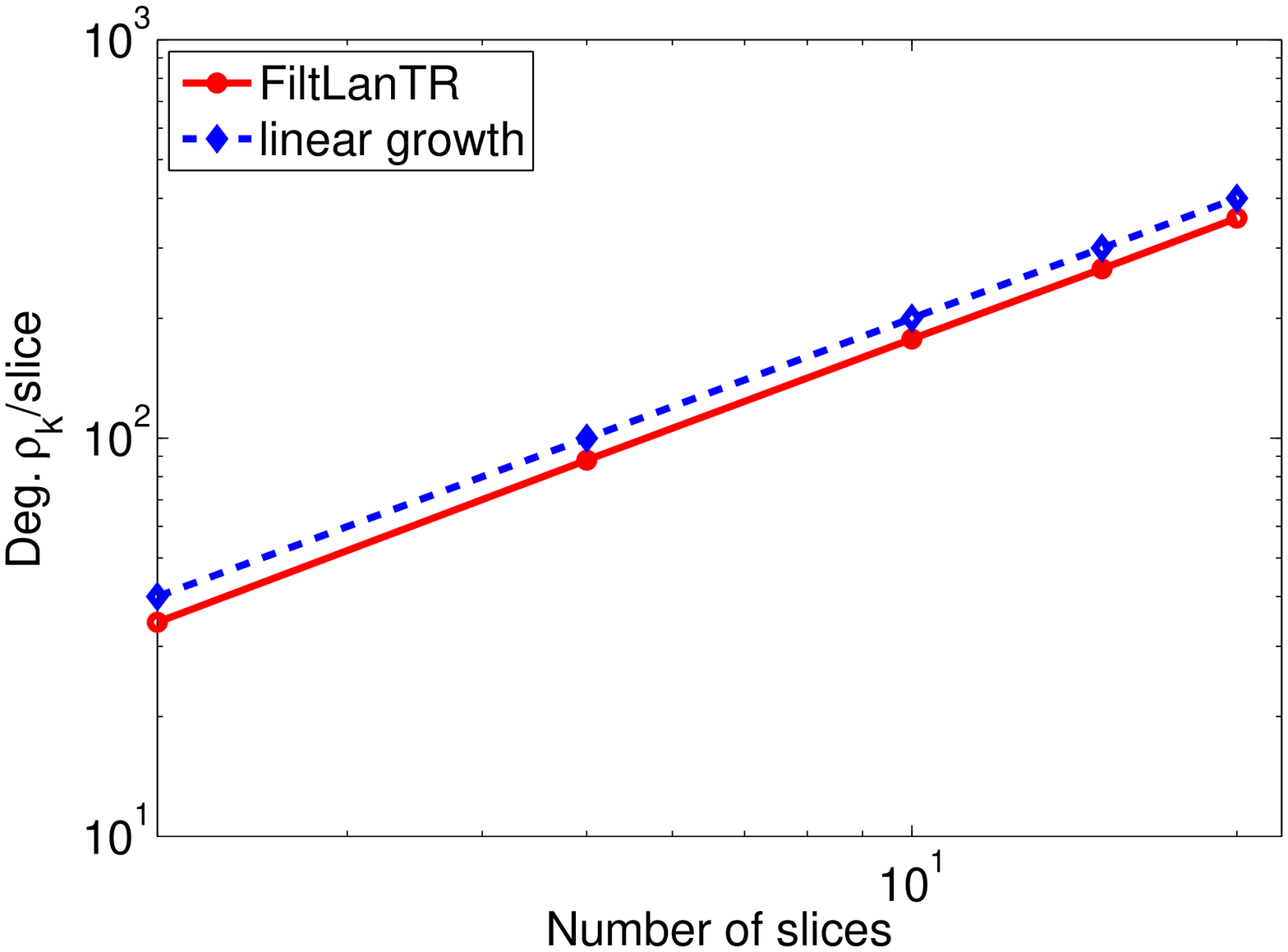}&
\includegraphics[width=0.46\textwidth]{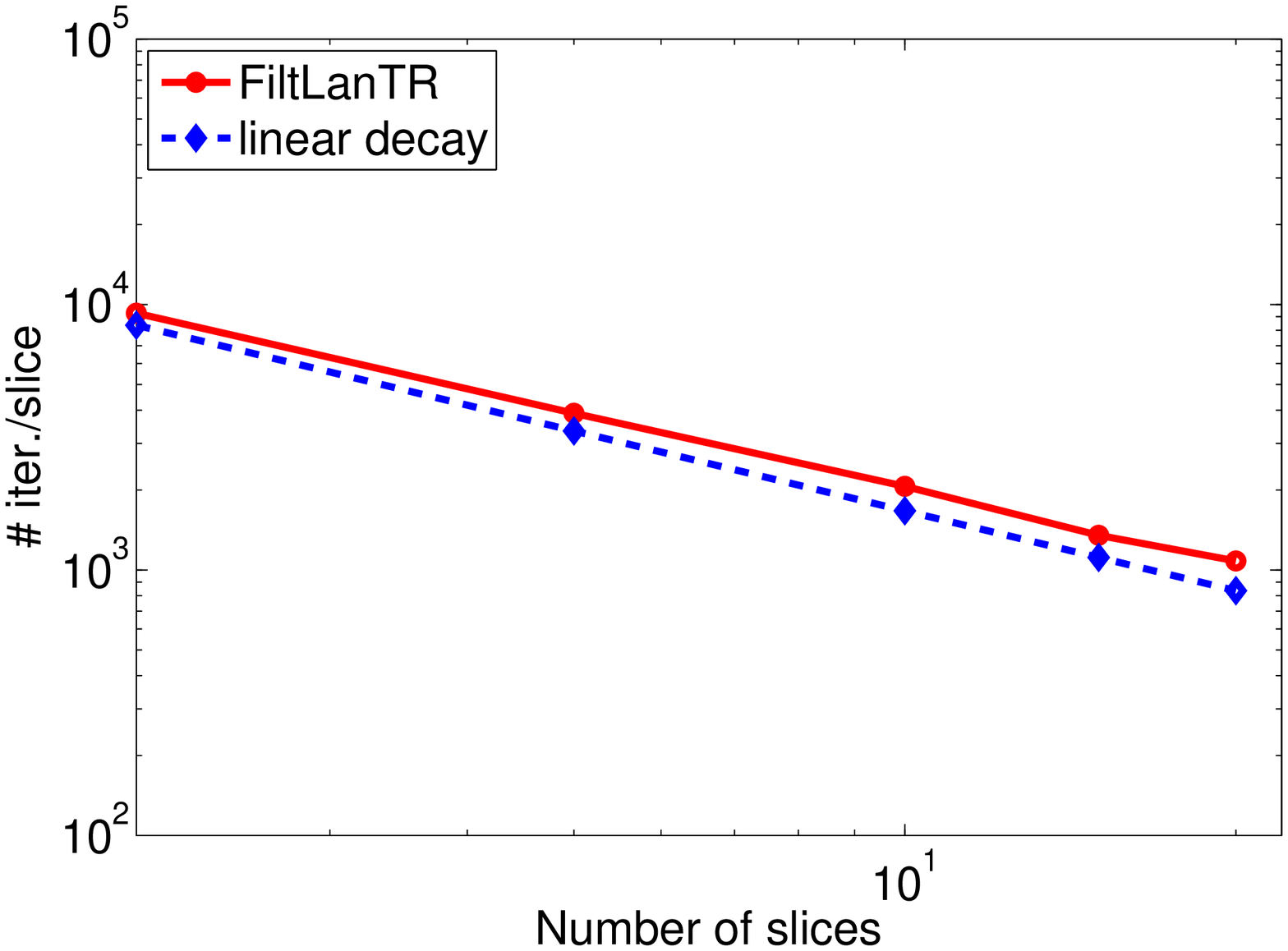} \\
\includegraphics[width=0.46\textwidth]{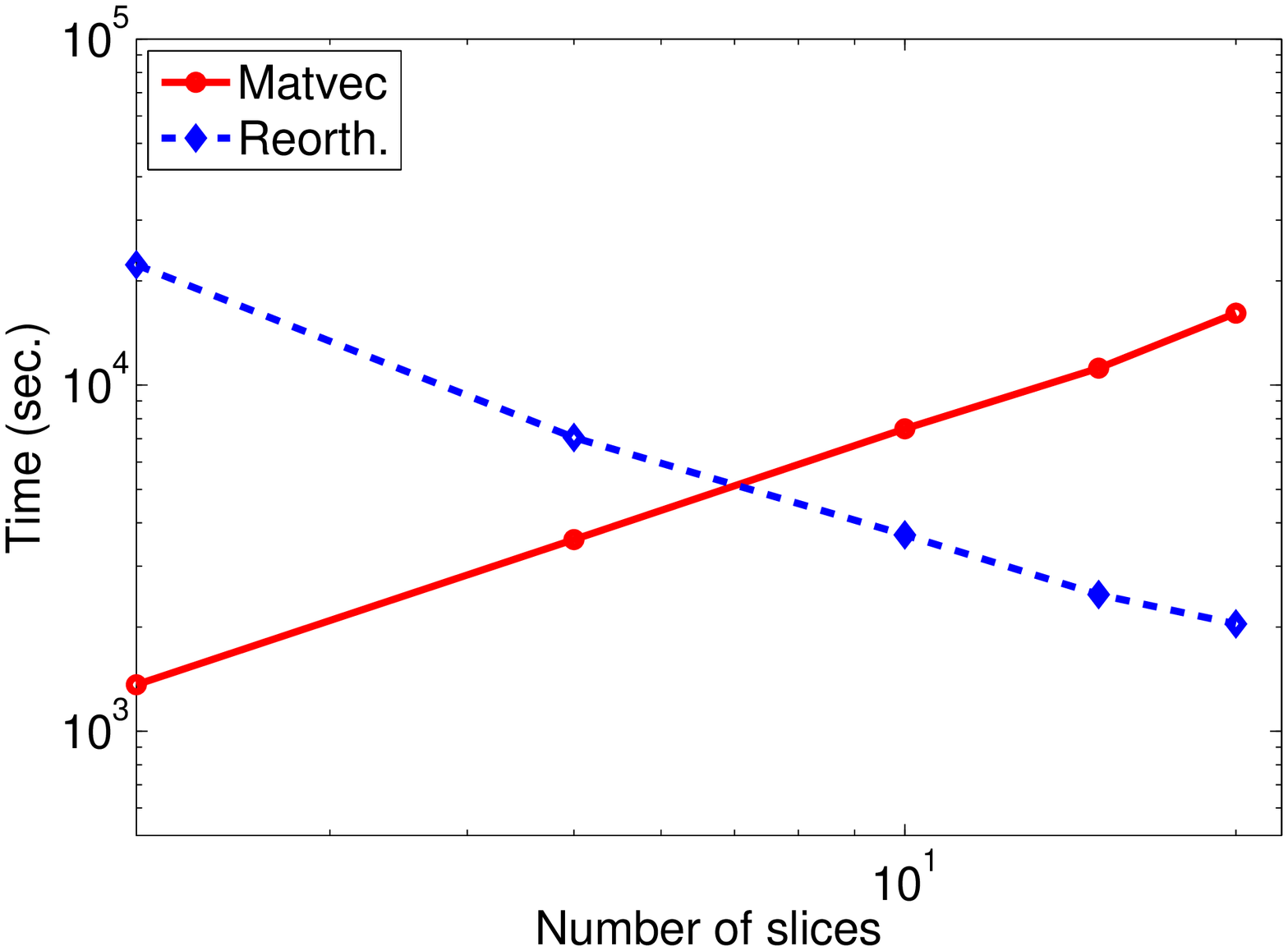}&
\includegraphics[width=0.46\textwidth]{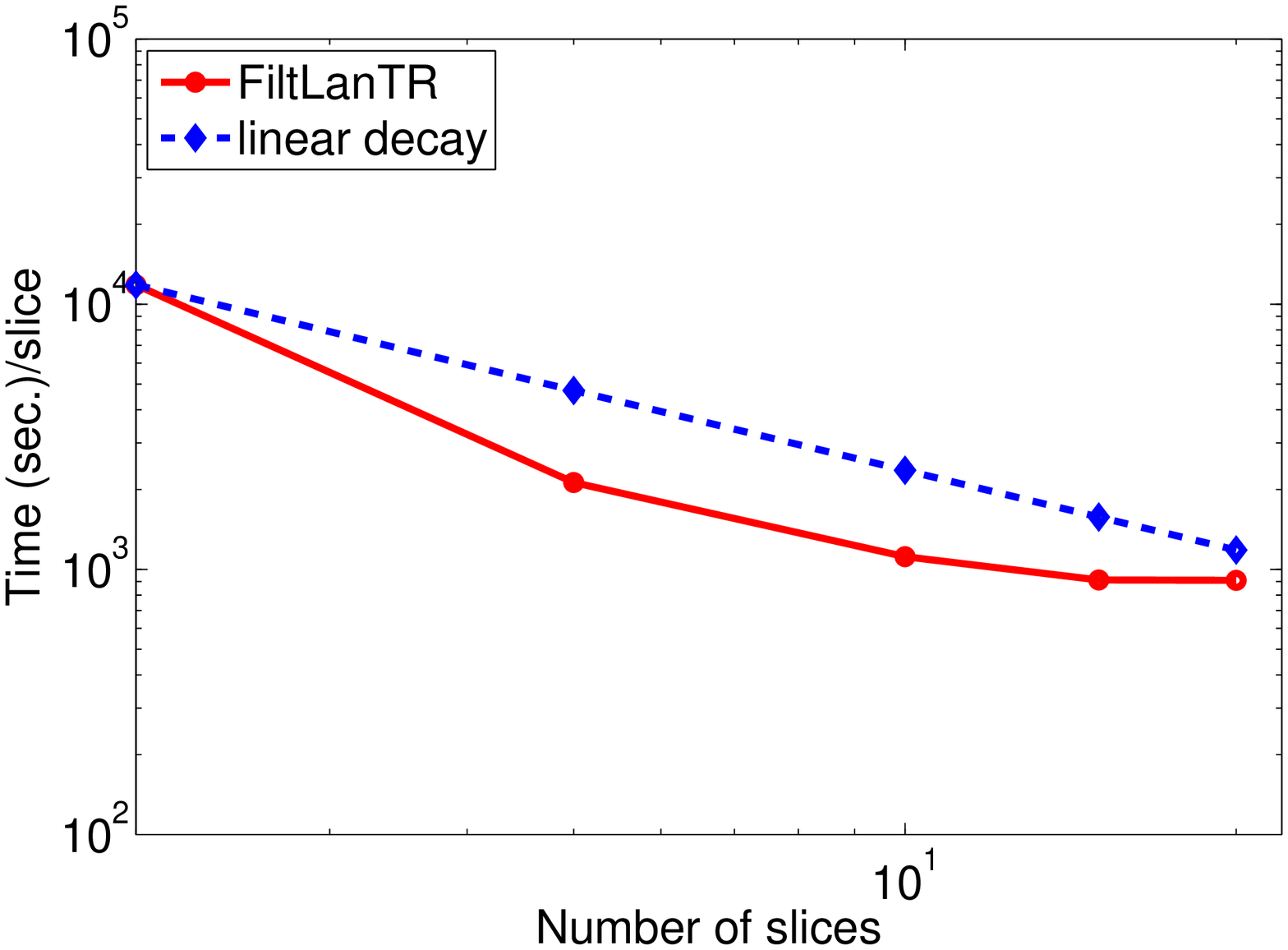} 
\end{tabular}
\caption{Trend plot of computed results in Table \ref{tab:Laplacianresult4} weighted by the number of slices.}
\label{fig:slice}
\end{figure}

\subsection{Matrices from electronic structure  calculations}\label{subsec:parsec}
In  this section we  compute eigenpairs  of five  Hamiltonian matrices
from electronic  structure calculations. These matrices are part of the
PARSEC set of the University of
Florida Sparse Matrix Collection and come from real space
discretizations of  Hamiltonians.  The matrix size $n$, the number
of nonzeros  $nnz$, the  range of the  spectrum $[a,\ b]$,  the target
interval \intv as well as  the exact number of eigenvalues inside this
interval  are reported  in  Table \ref{tab:Hamilton}.   Each \intv  is
selected to contain  the interval $[0.5n_0,\ 1.5n_0]$, where
$n_0$   corresponds   to  the   Fermi   level   of  each   Hamiltonian
\cite{Filtlan-paper}.
\begin{table}[tbh]
\caption{
Hamiltonian matrices from the 
PARSEC set in the University of
Florida Sparse Matrix Collection.}%
\label{tab:Hamilton}
\centering\tabcolsep6pt
\begin{tabular}
[c]{c|c|c|c|c|c}\hline
Matrix & $n$&$nnz$ & $[a,b]$ & $[\xi,\eta]$& \#eig\\\hline
$\mathrm{Ge_{87}H_{76}}$ & $112,985$ & $7,892,195$ & $[-1.214, 32.764]$ & $[-0.64, -0.0053]$ & 212 \\
$\mathrm{Ge_{99}H_{100}}$ & $112,985$ & $8,451,295$ & $[-1.226,32.703]$& $[-0.65, -0.0096]$&250 \\
$\mathrm{Si_{41}Ge_{41}H_{72}}$& $185,639$ & $15,011,265$ & $[-1.121,49.818]$&  $[-0.64,-0.0028]$& 218 \\
$\mathrm{Si_{87}H_{76}}$ &$240,369$ & $10,661,631$ & $[-1.196,43.074]$& $[-0.66,-0.0300]$& 213 \\
$\mathrm{Ga_{41}As_{41}H_{72}}$ & $268,096$ & $18,488,476$ & $[-1.250,1300.9]$&$[-0.64,-0.0000]$& 201\\\hline
\end{tabular}
\end{table}

The Hamiltonians under consideration have roughly $70$ nonzero entries per row 
and are much denser compared to the 3D Laplacian example. In fact, the number of nonzeros in their LU factors is around $n^2/5$ rendering
the classical Shift-and-invert Lanczos algorithm very ineffective. We then ran the experiments with \texttt{FiltLanTR} following the same settings as in the Laplacian example and reported the computational results in Table \ref{tab:Parsecresult2}. 

\begin{table}[tbh]
\caption{Numerical results for matrices in the PARSEC set with $\phi = 0.8$.}%
\label{tab:Parsecresult2}
\centering\tabcolsep5.0pt
\begin{tabular}
[c]{c|c|c|c|c|c|c|c}\hline
\multirow{2}{*}{Matrix}  & \multirow{2}{*}{deg} & \multirow{2}{*}{iter} & \multirow{2}{*}{matvecs}& \multicolumn{2}{|c|}{CPU Time (sec.)}  & \multicolumn{2}{|c}{Residual}\\\cline{5-8}
 &  & &  & matvec & Total  & Max & Avg  \\\hline
$\mathrm{Ge_{87}H_{76}}$ & 26 & 1431  &37482  & 282.70 & 395.91 &   $9.40\!\times\! 10^{-9}$ & $2.55\!\times\!10^{-10}$ \\
$\mathrm{Ge_{99}H_{100}}$ & 26 & 1615  & 42330  & 338.76 &  488.91 & $9.10\!\times\! 10^{-9}$ & $2.26\!\times\!10^{-10}$ \\
$\mathrm{Si_{41}Ge_{41}H_{72}}$ & 35 & 1420  & 50032  & 702.32 & 891.98 & $3.80\!\times\! 10^{-9}$ & $8.38\!\times\!10^{-11}$\\
$\mathrm{Si_{87}H_{76}}$ & 30 & 1427 & 43095  & 468.48 & 699.90 & $7.60\!\times\! 10^{-9}$ & $3.29\!\times\!10^{-10}$ \\
$\mathrm{Ga_{41}As_{41}H_{72}}$ & 202 & 2334 & 471669  & 8179.51 & 9190.46 & $4.20\!\times\! 10^{-12}$ & $4.33\!\times\!10^{-13}$ \\\hline
\end{tabular}
\end{table}

Since these Hamiltonians are  quite dense, the computational cost from
matrix-vector  products accounts for  a large  portion of  the overall
cost even when  low degree filters are in use.  This is illustrated in
the first  four tests. In addition, \texttt{FiltLanTR}  selects a much
higher    degree    polynomial     filter    for    the    Hamiltonian
$\mathrm{Ga_{41}As_{41}H_{72}}$  as compared  to the  others.  This is
because the  target interval $[-0.64,\ 0.0]$ is  quite narrow relative
to the range  of its spectrum $[-1.2502,\ 1300.93]$.  For this kind of
problems, it is  better to reduce the polynomial  degree by increasing
the  value of  $\phi$ slightly.  For example,  as shown  in Table
\ref{tab:Parsecresult3},   when  $\phi$   increases   to  $0.9$,   the
matrix-vector  product  time  drops  by  $30.45\%$  and  the  total
iteration       time      drops      by       $18.04\%$       for
$\mathrm{Ga_{41}As_{41}H_{72}}$.  However,  the  performance  for  the
other four  Hamiltonians deteriorates significantly.  Therefore, it is
only beneficial  to lower the degree  of the filter  for problems with
expensive matrix-vector  products and a high  degree polynomial filter
constructed by \texttt{FiltLanTR}.

\begin{table}[tbh]
\caption{Numerical results for matrices in the PARSEC set with $\phi = 0.9$.}%
\label{tab:Parsecresult3}
\centering\tabcolsep5.0pt
\begin{tabular}
[c]{c|c|c|c|c|c|c|c}\hline
\multirow{2}{*}{Matrix}  & \multirow{2}{*}{deg} & \multirow{2}{*}{iter} & \multirow{2}{*}{matvecs}& \multicolumn{2}{|c|}{CPU Time (sec.)}  & \multicolumn{2}{|c}{Residual}\\\cline{5-8}
 &  & &  & matvec & Total  & Max & Avg  \\ \hline
$\mathrm{Ge_{87}H_{76}}$ & 18 &1981  &36027  & 275.87 &462.96 & $9.20\!\times\! 10^{-9}$ & $5.15\!\times\!10^{-10}$ \\ 
$\mathrm{Ge_{99}H_{100}}$ & 19 & 2260& 43346  &351.68 & 603.72 & $7.10\!\times\! 10^{-9}$ & $2.83\!\times\!10^{-10}$\\ 
$\mathrm{Si_{41}Ge_{41}H_{72}}$ &24  &   1976 & 47818  & 679.19 & 971.19 & $8.10\!\times\!10^{-9}$ & $3.17\!\times\! 10^{-10}$\\ 
$\mathrm{Si_{87}H_{76}}$ &21 &3258 &68865 & 760.71 &1297.41 & $9.90\!\times \!10^{-9}$ & $7.78\!\times\! 10^{-10}$\\ 

$\mathrm{Ga_{41}As_{41}H_{72}}$ & 140&2334  &326961 &5688.46  &6703.93& $1.20\!\times \!10^{-10}$ & $2.03\!\times \!10^{-12}$\\ \hline
\end{tabular}
\end{table}

\subsection{Parallel results with OpenMP}\label{subsec:openmp}
As has been demonstrated on the 3D Laplacian example in Section \ref{subsec:lap3d},
\texttt{FiltLanTR} can be naturally utilized  
within the divide and conquer strategy for computing a large number of eigenpairs.
The ability to target different parts of spectrum independently opens an additional level of concurrency, which can be efficiently exploited 
by modern parallel machines.  

In this section, we demonstrate the scalability potential of this divide and 
conquer  approach by simply adding a naive OpenMP parallelization across 
different spectral intervals.
In each interval, \texttt{FiltLanTR} 
is invoked, so that different parts of spectrum are handled independently by concurrent 
threads ran on different computational cores.      

As a test problem, we computed $1002$ lowest eigenpairs of 
another matrix from the PARSEC set discussed earlier. 
This is the matrix  $\mathrm{SiO}$ which has a size 
$n = 33,401$ and $nnz = 1,317,655$ nonzeros.
\begin{figure}[htb]
\centering 
\includegraphics[width=0.56\textwidth]{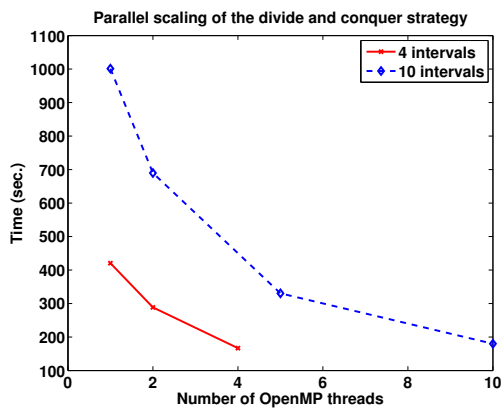}
\caption{An OpenMP parallelization across 4 and 10 spectral intervals of a divide and conquer 
approach for SiO matrix; $n = 33,401$ and $nnz = 1,317,655$.}
\label{fig:sio_omp}
\end{figure}

Figure~\ref{fig:sio_omp} demonstrates parallel scaling of such a divide and conquer approach for computing 
$1002$ lowest eigenpairs using 4 and 10 spectral intervals. It can be seen that solution time is 
decreased by roughly the same factor by which the number of OpenMP threads (or, equivalently, computational
cores) is increased, suggesting a near-optimal scaling of the entire strategy. Note that our test only
illustrates parallelism achieved by independently treating different spectral regions.  
The second level of parallelism available in the re-orthogonalization 
process and matrix-vector products  as shown in Figure \ref{fig:evsl} will be addressed in our future work. 


As has been observed in Section \ref{subsec:lap3d}, the choice of the number of intervals has a strong impact on 
the overall performance of \texttt{FiltLanTR}, with the recommended number of eigenvalues per slice around $200$ to $300$. The test in Figure~\ref{fig:sio_omp} reaffirms this finding. It shows that faster solution times 
are obtained if $4$ intervals, with approximately $250$ eigenvalues per slice, are used as opposed to
using $10$ intervals with roughly $100$ eigenvalues in each sub-interval.

\section{Conclusion}\label{sec:con}
One of the critical components of a spectrum slicing algorithm 
designed to compute many eigenpairs of a large sparse symmetric
matrix is a procedure that can be used to compute eigenvalues contained
in a sub-interval $[\xi,\ \eta]$. We developed an efficient way to accomplish
this task. Our algorithm is based on combining polynomial filtering 
with a Thick-Restart Lanczos algorithm. Thick restarting is employed
to limit the cost of orthogonalization. 
The polynomial filter that maps eigenvalues within $[\xi,\ \eta]$
to eigenvalues with the largest magnitude of the transformed problem,
is obtained from a least-squares approximation to an
 appropriately centered Dirac-$\delta$ distribution. Numerical experiments
show that
such a construction yields effective polynomial filters which along with a
Thick-Restart Lanczos procedure enable desired eigenpairs 
to be computed efficiently.

\section*{Acknowledgments}
YS had stimulating discussions with Jared Aurentz which lead to an improved
scheme for locating the correct placement of the Delta Dirac function in the interval 
$[a, \ b]$. 

\bibliographystyle{siam}
\bibliography{local.bib}

\clearpage 

\section*{Appendix}
\renewcommand{\thesubsection}{\Alph{subsection}}
\subsection{Balancing the filter via an eigenvalue problem }\label{appdix:1}
Another way to balance the filter is via the solution of an
 eigenvalues problem involving a  $k\times k$ Hessenberg matrix.
Let $t_j = \cos(j\theta_\gamma)= T_j(\gamma)$, for $j=0,\cdots, k-1$. Then the 3-term recurrence of Chebyshev polynomials yields
\eq{eq:Neq8}
2\gamma \times t_j = \left\{ \begin{array}{lll}
2 t_{j+1} & \mbox{if} & j=0 \\
t_{j+1}+ t_{j-1} & \mbox{if} & j>0 \end{array} \right. .
\en
For $\gamma$ that is a solution of Equation (\ref{eq:Neq2}) we have:
\eq{eq:Neq3}
t_k= -\sum_{j=0}^{k-1}  \beta_j  t_j ,
\qquad \hbox{with} \qquad 
\beta_j = 
\frac{g_j^k [\cos(j\theta_\xi)-   \cos(j\theta_\eta)]}
{g_k^k [\cos(k\theta_\xi)-   \cos(k\theta_\eta)]}
.
\en 
Thus, when $j=k-1$ in (\ref{eq:Neq8}) the following equation holds
\eq{eq:Neq5}
2\gamma \times t_{k-1} = -\sum_{j=0}^{k-1}  \beta_j  t_j +t_{k-2}.
\en 

As a result, denoting by $\mathbf{t}$ the vector with components $t_j, j=0,\ldots,k-1$ relations (\ref{eq:Neq8}) and
(\ref{eq:Neq5}) can be expressed in the form: 
\eq{eq:eig}
 H \mathbf{t} = 2\gamma \mathbf{t},
\en
where
\[
H = 
\left[
\begin{array}{cccccc}
0  &  2    &           &         &          &       \\
1  &  0    & 1        &         &          &       \\
    &   1  &   0       &    1   &          &       \\
    &       &   \ddots  &  \ddots & \ddots   &       \\
    &       &       &   1      &  0       &    1    \\
- \beta_0   &  -\beta_1     &  \ldots     &  \ldots     &   1-\beta_{k-2}      & -\beta_{k-1}       
\end{array}
\right].
\]
This shows  that $\gamma$  is an eigenvalue  of the  Hessenberg matrix
$H/2$. We  can take  the eigenvalue  of $H/2$ that  is closest  to the
value $\cos^{-1}(\theta_c)$ among  those belong to $[\xi,  \ \eta]$ as
the center. 

\subsection{Additional numerical results for the 3D discrete Laplacian}
\label{sec:differentslice}

Results for a $3$D Laplacian with $2$, $5$, $15$ and $20$ slices are tabulated in  Table~\ref{tab:Laplacianresult4}.

\begin{table}[tbh]
\caption{Numerical results for a 3D  Laplacian  with $\phi = 0.8$ and 
different numbers of slices.}%
\label{tab:Laplacianresult4}
\centering\tabcolsep4.5pt
\begin{tabular}
[c]{c|c|c|c|c|c|c|c}\hline
\multirow{2}{*}{\#slices}  & \multirow{2}{*}{deg} & \multirow{2}{*}{iter} & \multirow{2}{*}{matvecs}& \multicolumn{2}{|c|}{CPU Time(sec.)}  & \multicolumn{2}{|c}{residual}\\\cline{5-8}
 & & & & matvecs & total & max & avg \\\cline{1-8}
\multirow{2}{*}{2} & 28& 8263 & 233225  & 486.12 & 10857.38 & $3.70\!\times \!10^{-9}$&$2.99\!\times\! 10^{-11}$\\
& 41 & 10306 & 424439 & 877.36& 12777.31& $7.50\!\times\! 10^{-9}$&$1.05\!\times\! 10^{-10}$\\\hline
 &62& 3376 & 210081  & 435.31 &1657.48& $9.90\!\times \!10^{-9}$&$1.90\!\times\! 10^{-10}$\\
& 76 & 4238 & 322975 & 667.70& 2212.86& $9.10\!\times\! 10^{-9}$&$5.98\!\times\! 10^{-11}$\\
5& 88 & 3396 & 299665 & 616.87& 1865.45& $7.50\!\times\! 10^{-9}$&$8.32\!\times\! 10^{-11}$\\
& 100 & 4235 & 424452 & 872.03& 2401.22& $9.60\!\times\! 10^{-9}$&$1.74\!\times\! 10^{-10}$\\
& 114 & 4214 & 481350 & 988.00& 2497.86& $5.70\!\times\! 10^{-9}$&$5.14\!\times\! 10^{-11}$\\\hline
 &170& 1251 & 213074  & 440.03 &592.20& $7.30\!\times \!10^{-9}$&$1.23\!\times\! 10^{-10}$\\
&185 & 1239 & 229654 & 474.06&626.68& $7.50\!\times\! 10^{-11}$&$7.43\!\times\! 10^{-13}$\\
 & 198 & 1248 & 247529 & 511.41& 660.84& $4.30\!\times\! 10^{-12}$&$1.42\!\times\! 10^{-13}$\\
& 215 & 1540 & 331780 & 684.47& 872.09& $4.60\!\times\! 10^{-11}$&$2.99\!\times\! 10^{-13}$\\
& 229 & 1248 & 286260 & 590.86& 743.51& $5.20\!\times\! 10^{-9}$&$8.91\!\times\! 10^{-11}$\\
 &238& 1247 & 297264  & 612.96 &765.26& $3.80\!\times \!10^{-9}$&$4.40\!\times\! 10^{-11}$\\
& 250 & 1573 & 394002 & 812.99&1009.45& $4.70\!\times\! 10^{-9}$&$7.35\!\times\! 10^{-11}$\\
15& 267 & 1246 &333197& 686.88& 840.87& $2.70\!\times\! 10^{-9}$&$2.39\!\times\! 10^{-11}$\\
& 280 & 1250 & 350490 & 722.35& 869.20& $5.40\!\times\! 10^{-11}$&$1.42\!\times\! 10^{-12}$\\
& 294 & 1573 & 463311 & 955.36& 1154.58& $5.40\!\times\! 10^{-11}$&$1.42\!\times\! 10^{-12}$\\
&304&1250 & 380523  & 784.04 &933.47& $1.60\!\times \!10^{-9}$&$2.82\!\times\! 10^{-11}$\\
& 313 & 1573 & 493227 & 1016.69& 1213.15& $5.80\!\times\! 10^{-9}$&$4.99\!\times\! 10^{-11}$\\
& 323 & 1251 &  404615 & 834.10& 983.25& $2.00\!\times\! 10^{-9}$&$2.61\!\times\! 10^{-11}$\\
& 333 & 1570 & 523746 & 1078.74& 1279.77& $8.90\!\times\! 10^{-9}$&$1.40\!\times\! 10^{-10}$\\
& 392& 1219 & 478463 & 985.61& 1128.78& $8.10\!\times\! 10^{-10}$&$1.59\!\times\! 10^{-11}$\\ \hline
&229& 952 & 218422  & 456.02 &545.01& $8.8
0\!\times \!10^{-9}$&$1.69\!\times\! 10^{-10}$\\ 
&243 & 954 & 232229 & 485.25&570.54& $1.70\!\times\! 10^{-10}$&$4.88\!\times\! 10^{-12}$\\
 & 253 & 1188 & 301266 & 628.85&  741.02& $9.10\!\times\! 10^{-9}$&$1.54\!\times\! 10^{-10}$\\
& 268 & 1009 & 270837 & 565.15& 659.90& $4.20\!\times\! 10^{-9}$&$4.26\!\times\! 10^{-11}$\\
& 283 & 1186 & 336397 & 702.43& 813.16& $2.30\!\times\! 10^{-9}$&$1.37\!\times\! 10^{-11}$\\
&294& 1218 &  358858 & 748.57&861.85& $1.50\!\times \!10^{-9}$&$2.99\!\times\! 10^{-11}$\\
&310 & 1016 & 315444 & 658.70&758.21& $1.70\!\times\! 10^{-9}$&$4.19\!\times\! 10^{-11}$\\
& 321 & 949 & 305134 & 636.30& 724.47& $4.10\!\times\! 10^{-10}$&$6.14\!\times\! 10^{-12}$\\
& 333 & 1243 & 414780 & 865.01& 986.00& $6.10\!\times\! 10^{-9}$&$7.71\!\times\! 10^{-11}$\\
20& 345 & 984 & 339983 & 709.70& 800.04& $9.80\!\times\! 10^{-9}$&$1.33\!\times\! 10^{-10}$\\
  &357& 944 & 337564  & 703.77 &792.95& $9.80\!\times \!10^{-9}$&$1.33\!\times\! 10^{-10}$\\
& 369 & 983 & 363258 & 757.46&847.84& $9.80\!\times\! 10^{-9}$&$1.33\!\times\! 10^{-10}$\\
& 382 & 1207 &462063& 961.42& 1081.01& $2.40\!\times\! 10^{-9}$&$1.69\!\times\! 10^{-11}$\\
& 395 & 982 & 388442 & 807.92& 897.06& $6.50\!\times\! 10^{-9}$&$1.04\!\times\! 10^{-10}$\\
& 408 & 1174 & 480024 & 998.82&1110.60& $5.20\!\times\! 10^{-10}$&$1.22\!\times\! 10^{-11}$\\
&422&1014 & 428491  & 891.10 &987.13& $5.90\!\times \!10^{-10}$&$1.11\!\times\! 10^{-11}$\\
& 426 & 1215 &518630 & 1079.17& 1193.59& $6.60\!\times\! 10^{-9}$&$9.95\!\times\! 10^{-11}$\\
& 440 & 1246 &  549306 & 1142.45& 1262.09& $6.60\!\times\! 10^{-9}$&$6.67\!\times\! 10^{-11}$\\
& 455 & 1215 & 553931 & 1152.21& 1268.52& $3.20\!\times\! 10^{-9}$&$3.50\!\times\! 10^{-11}$\\
& 603& 955 & 576607 & 1198.94 & 1281.51& $2.20\!\times\! 10^{-9}$&$5.20\!\times\! 10^{-11}$\\\hline
\end{tabular}
\end{table}

%



\end{document}